\documentclass[10pt]{article}

\usepackage{amsmath}

\numberwithin{equation}{section}
\usepackage{amsfonts}
\usepackage{amsthm}
\usepackage{amssymb}
\usepackage{bbm}
\usepackage{centernot}
\usepackage[x11names]{xcolor}
\usepackage{chemfig}
\usepackage{amsmath}
\usepackage{cite}
\usepackage[nodayofweek]{datetime} 
\usepackage{dsfont}
\usepackage{enumitem}
\usepackage{euscript}
\usepackage{faktor}
\usepackage{hyperref}
\usepackage{float}
\usepackage{geometry}
\usepackage{graphicx}
\usepackage{mathrsfs}
\usepackage{mathtools}
\usepackage{polynom}
\usepackage{stmaryrd}
\usepackage{tikz}
\usepackage{tikz-cd}
\usepackage{wasysym}
\usepackage{xfrac}
\usepackage[x11names]{xcolor}
\usepackage{mathtools}
\usepackage{setspace}
\usepackage{verbatim}

\def\acts{\curvearrowright}

\newif\ifproofread

\DeclarePairedDelimiter\abs{\lvert}{\rvert}
\DeclarePairedDelimiter\norm{\lVert}{\rVert}

\makeatletter
\newcommand*\bigcdot{\mathpalette\bigcdot@{.5}}
\newcommand*\bigcdot@[2]{\mathbin{\vcenter{\hbox{\scalebox{#2}{$\m@th#1\bullet$}}}}}
\makeatother

\usepackage{mathtools}

\DeclarePairedDelimiter\floor{\lfloor}{\rfloor}

\AtBeginEnvironment{align}{\setcounter{equation}{0}}

\newtheorem*{Lemma*}{Lemma}
\newtheorem*{Corollary*}{Corollary}
\newtheorem*{theorem*}{Theorem}
\newtheorem*{definition*}{Definition}

\newtheorem{theorem}{Theorem}[section]
\newtheorem{nota}{Notation}[section]
\newtheorem{definition}[theorem]{Definition}
\newtheorem{remark}[theorem]{Remark}
\newtheorem{Lemma}[theorem]{Lemma}

\newtheorem{Corollary}[theorem]{Corollary}

\newcommand{\thistheoremname}{}
\newtheorem*{genericthm}{\thistheoremname}

\newcommand{\CC}{\mathbb{C}}

\newcommand{\FF}{\mathbb{F}}

\newcommand{\NN}{\mathbb{N}}

\newcommand{\QQ}{\mathbb{Q}}

\newcommand{\RR}{\mathbb{R}}

\newcommand{\ZZ}{\mathbb{Z}}

\newcommand{\eps}{\epsilon}

\newcommand{\on}[1]{\operatorname{#1}}

\newcommand{\inner}[2]{\left\langle#1,#2\right\rangle}

\newcommand{\inn}[1]{\left<#1\right>}

\newcommand{\cov}[1]{\text{cov}(#1)}

\newcommand{\pid}{\mathrel{\ooalign{$\lneq$\cr\raise.22ex\hbox{$\lhd$}\cr}}}
\newcommand{\colf}[4]{\begin{bmatrix}
           #1 \\
           #2 \\
           #3 \\
           #4
         \end{bmatrix}}

\AtBeginDocument{\renewcommand{\qedsymbol}{$_\blacksquare$}}

\title{The Equidistribution of Grids of Rings of Integers in Number Fields of Degrees 3,4 and 5}
\author{Yuval Yifrach}
\Huge
\begin{document}
\maketitle
\begin{abstract}
It was shown by M. Bhargava and P. Harron that for $n=3,4,5$, the shapes of rings of integers of $S_n$-number fields of degree $n$ become equidistributed in the space of shapes when the fields are ordered by discriminant. Instead of shapes, we correspond grids to each number field, which preserve more of the number fields' data. The space of grids is a fiber bundle over the space of shapes. We strengthen Bhargava-Harron's result by proving that the grids of rings of integers of $S_n$-number fields become equidistributed in the space of grids.
\end{abstract}
\section{Introduction}\label{Intro}
Let $\FF$ be a number field of degree $n$ and let $\mathcal O_{\FF}$ denote its ring of integers. Let $\sigma_1,\dots,\sigma_r$ denote the $r$ real embeddings of $\FF$ and $\tau_1,\overline{\tau_1},\dots,\tau_s,\overline{\tau_s}$ denote the $s$ pairs of complex embeddings of $\FF$. We define the signature of $\FF$ to be $(r,s)$. 
These embeddings induce a map $\sigma:\FF\rightarrow \RR^r\oplus\CC^s\cong \RR^n$ defined by:
\begin{equation*}
	\FF\ni\alpha\mapsto (\sigma_1(\alpha),\dots,\sigma_r(\alpha),\tau_1(\alpha),\dots,\tau_s(\alpha))\in \RR^r\oplus\CC^s.
\end{equation*}
We denote the image of $\mathcal O_{\FF}$ under $\sigma$ by $\Lambda(\FF)$.
The space $X_n$ of unimodular lattices in $\RR^n$ is commonly identified with $\on{SL}_n(\RR)/\on{SL}_n(\ZZ)$ and is endowed with the quotient topology. Given a lattice $\Lambda$ in $\RR^n$ we denote $\inn{\Lambda}$ to be the corresponding scaled unimodular lattice. It is well known that $\Lambda(\FF)$ is a lattice in $\RR^n$ and therefore $\inn{\Lambda(\FF)}\in X_n$.

It is very natural to study geometric aspects of the lattices $\Lambda(\FF)$ as $\FF$ varies. It is natural to expect that unless an obvious obstacle of an algebraic origin exists, various geometric aspects of the lattices $\Lambda(\FF)$ would behave "randomly" as $\FF$ varies. The first naive attempt is to consider $\inn{\Lambda(\FF)}$ as a point in $X_n$ equipped with the uniform ($\on{SL}_n(\RR)$-invariant) probability measure $m_{X_n}$, and expect that the countable collection $\inn{\Lambda(\FF)}$ equidistributes there as $\FF$ varies in a natural way. This naive attempt fails drastically because of an obvious "algebraic" reason as we shall presently explain. For convenience, assume that the signature of $\FF$ is $(n,0)$. Then the unimodular lattice $\inn{\Lambda(\FF)}$ contains the vector $v_{\overline 1}=\left(\frac{1}{\abs{\on{Disc}(\FF)}}\right)^{\frac{1}{2n}}(1,\dots,1)$ and therefore by Mahler's compactness criterion (see ~\cite{Mah}), we have that not only does $\Lambda(\FF)$ fail to equidistribute as $\FF$ varies, but in fact, any compact set in the space of lattices contains only finitely many such lattices (in other words $\inn{\Lambda(\FF)}$ diverges to infinity as $\FF$ varies). 

In the paper ~\cite{BhaPip}, the authors show, roughly, that orthogonal to the short vector $v_{\overline 1}$, the geometry of the lattice $\inn{\Lambda(\FF)}$ does look random as $\FF$ varies. In this paper we  consider slightly more refined ways to remedy the divergence to infinity. That is, we identify a closed sub-homogeneous space of $X_n$, which we call $Y^{\ell_{\overline 1}}_1$ (the reason for this notation will become clear later), and deform $\inn{\Lambda(\FF)}$ into $Y_1^{\ell_{\overline 1}}$ in a manner that still preserves some of the geometry of $\Lambda(\FF)$, and establish the equidistribution of the resulting points in $Y^{\ell_{\overline 1}}_1$ as $\FF$ varies. 

In fact, the deformation of $\inn{\Lambda(\FF)}$ into $Y^{\ell_{\overline 1}}_1$ will depend on an ad hoc choice of a linear direct complement $E$ of the line $\RR\cdot (1,\dots,1)\subset \RR^n$ (in the case of signature $(n,0)$, for the remaining signatures the vector will be replaced by a different one) and it is interesting to note that the answer to the question "how do the deformed lattices distribute in $Y^{\ell_{\overline 1}}_1$ as $\FF$ varies" depends on the choice of $E$, namely, the asymptotic behaviour is completely different for $E = V_0$ and for $E\neq V_0$ where $V_0$ is the orthogonal complement of $(1,\dots,1)$.

We remark that (as will be explained in more detail below), the deformed lattices in $Y^{\ell_{\overline 1}}_1$ record some of the original geometrical information of $\Lambda(\FF)$. For example, we show below that one can recover the main result of ~\cite{BhaPip} from our main results. As we are building heavily on the analysis and results from ~\cite{BhaPip}, this last remark should only be taken into account on the motivational level, supporting the "natural" nature of our deformation procedure which we shall now define.

For any line $\ell$ in $\RR^n$ denote:
\begin{equation}
	Y^{\ell}\vcentcolon=\{\Lambda\in X_n:\Lambda\cap \ell\text{ is a lattice in }\ell\},\text{ }Y_1^{\ell}\vcentcolon=\{\Lambda\in X_n:\Lambda\cap \ell\text{ is a unimodular lattice in }\ell\}.
\end{equation}
\begin{definition}\label{homMeasure}
A Borel probability measure $\mu$ on $X_n$ is called homogeneous if there exists a closed subgroup $H\leq \on{SL}_n(\RR)$ such that $\mu$ is invariant under $H$ and supported on an $H$ orbit. An orbit supporting a homogeneous measure is called periodic. Periodic orbits are closed in $X_n$.
\end{definition}
Each of the subsets $Y^{\ell},Y^{\ell}_1$ is an orbit of a closed subgroup of $\on{SL}_n(\RR)$. Indeed, $Y^{\ell}$ is the orbit of the subgroup of elements $g\in \on{SL}_n(\RR)$ satisfying $g\ell=\ell$ and $Y^{\ell}_1$ is the orbit of the subgroup of elements $g\in \on{SL}_n(\RR)$ satisfying $gv=v$ for every $v\in \ell$. The orbit $Y^{\ell}$ is dense in $X_n$ and the orbit $Y^{\ell}_1$ is closed in $X_n$ and in fact periodic. We denote the homogeneous measure on $Y^{\ell}_1$ by $m_{Y^{\ell}_1}$.
 We will use the following normalization group to deform elements of $Y^{\ell}$ into $Y^{\ell}_1$:
\begin{definition}\label{vagner}
	Given subspaces $E_1,E_2\subset \RR^n$ such that $E_1\oplus E_2=\RR^n$ and $t\in \RR$, we define a linear operator $g_{E_1,E_2}^t\in \operatorname{SL}_n(\RR)$ by $g(x)=e^{t/dim(E_1)}x$, $g(y)=e^{-t/dim(E_2)}y$ for every $x\in E_1,y\in E_2$. The set $\{g_{E_1,E_2}^t\}_{t\in \RR}$ is a group, which we shall call the {$E_1,E_2$ normalization group}.
\end{definition}

Define the vector $\overline 1\vcentcolon=(1,\dots,1,1,0,\dots,1,0)\in \RR^n$ as the image of $(1_{\RR},\dots,1_{\RR},1_{\CC},\dots,1_{\CC})\in \RR^r\oplus \CC^s$ under the identification with $\RR^n$. Then $\overline 1$ is an element of $\Lambda(\FF)$ because it is the image of the identity in $\FF$ under $\sigma$. Denote $\ell_{\overline 1}=\{t\overline 1:t\in \RR\}$ and note that $\inn{\Lambda(\FF)}\in Y^{\ell_{\overline 1}}$. Let $\inner{\cdot}{\cdot}$ denote the standard inner product on $\RR^n$. For any subspace $E\subset \RR^n$, define $E^{\perp}=\{x\in \RR^n:\inner{x}{E}=0\}$ and $V_0=\ell_{\overline 1}^{\perp}$.

Let $E$ be a subspace of $\RR^n$ of co-dimension 1 satisfying $E\oplus \ell_{\overline 1}=\RR^n$ and let $t\in \RR$ be such that:
\begin{equation}\label{orangeJuice}
	q_E(\FF)\vcentcolon=g^t_{E,\ell_{\overline 1}}\inn{\Lambda(\FF)}\in Y_1^{\ell_{\overline 1}}
\end{equation}
($t$ will be the solution of $e^{-t}\operatorname{cov}(\Lambda(\FF))^{-1/n}\sqrt{n}=1$). Therefore $q_{E}(\cdot)$ deforms an element of $Y^{\ell_{\overline 1}}$ into an element of $Y_1^{\ell_{\overline 1}}$.
The following definition parametrizes a subset of the number fields using the discriminant. 
\begin{definition}
	For any $T>0$ and signature $(n-2s,s)$, let $\mathcal F^s_n(T)$ be the set of isomorphism classes of degree $n$ number fields $\FF$ of signature $(n-2s,s)$, with discriminant $\leq T$ in absolute value such that the Galois group of the normal closure of $\FF$ is $S_n$. It is well known that $\mathcal F_n(T)$ is finite for any $T>0$.
\end{definition}
So far we have constructed a correspondence between number fields and elements of $Y_1^{\ell_{\overline 1}}$, which are essentially lattices with some extra properties. As appeared earlier, the aim of this paper is to analyze the connection between number fields and lattices given by this correspondence. More concretely, we will analyze the asymptotic distribution as $T\rightarrow \infty$ of $\{q_E(\FF)\}_{\FF\in \mathcal F_n^s(T)}$ for the different choices for $E$. The answer will turn out to be dependent upon the choice of $E$ and to imply Bhargava-Harron's main result in ~\cite{BhaPip}.

\subsection{New Equidistribution Results for Number Fields of Degrees 3,4,5}
In this subsection we state our new results which implement the philosophy described in the first part of this introduction.
\begin{theorem}\label{unif3} 
Let $n = 3,4,5$ and let $(r,s)$ be a signature. Let $E\subset\RR^n$ be a subspace such that $E\oplus \ell_{\overline 1}=\RR^n$ and $E\neq V_0$. Then:
\begin{equation}
	\frac{1}{\abs{\mathcal F^s_n(T)}}\sum_{\FF\in \mathcal F^s_n(T)}\delta_{q_E(\FF)}\longrightarrow m_{Y_1^{\ell_{\overline 1}}}\text{ weakly as }T\longrightarrow \infty.
\end{equation}
\end{theorem}
\begin{theorem}\label{torsion2}
	Let $n=3,4,5$ and let $(n-2s,s)$ be a signature. Denote $\on{Div}(n)=\{1\leq k\leq n:k\mid n\}$ to be the set of divisors of $n$. Then for any $k\in \on{Div}(n)$ there exists a periodic orbit $F^{(k)}\subset Y_1^{\ell_{\overline 1}}$, $\ell=\ell(k)\in \NN$, disjoint closed connected subsets $F^{(k)}_1,\dots,F^{(k)}_{\ell(k)}\subset F^{(k)}$ and nonnegative scalars $s_{k,\ell};\ell=1,\dots,\ell(k)$ not all zero, such that:
 \begin{equation}
	\frac{1}{\abs{\mathcal F^s_n(T)}}\sum_{\FF\in \mathcal F^s_n(T)}\delta_{q_{V_0}(\FF)}\longrightarrow \sum_{k\in \on{Div}(n)}\sum_{\ell=1}^{\ell(k)}s_{k,\ell}\mu^{(k)}\mid_{F^{(k)}_{\ell}}\text{ weakly as }T\longrightarrow \infty.
\end{equation}
where $\mu^{(k)}$ is the homogenous measure supported on the periodic orbit $F^{(k)}$.
\end{theorem}
The two Theorems above are the main results of this paper. In order to motivate them we explain how each of them implies Bhargava-Harron's main result in ~\cite{BhaPip} which is stated below as Theorem \ref{BhaPipMain}. 

First we give additional background required to state ~\cite[Theorem 1]{BhaPip}. In the above setting, consider the orthogonal projection of $\Lambda(\FF)$ on $V_0$ and denote it by $P\Lambda(\FF)$. Fix a linear isometry $\phi$ between $V_0$ and $\RR^{n-1}$ and use it to identify $P\Lambda(\FF)$ with its image under $\phi$. From now on we think of $P\Lambda(\FF)$ as a subset of $\RR^{n-1}$. Since $\overline 1\in \Lambda(\FF)$, $P\Lambda(\FF)$ is a lattice in $\RR^{n-1}$. The shape of $\FF$, denoted $s(\FF)$, is defined as the $\on{O}_{n-1}(\RR)$-orbit of the lattice $\inn{P\Lambda(\FF)}$ inside $X_{n-1}$. It is thus an element of the double quotient $\on{O}_{n-1}(\RR)\setminus \on{SL}_{n-1}(\RR)/\on{SL}_{n-1}(\ZZ)$. This double quotient is called the space of $n-1$-dimensional shapes and is denoted $\mathcal S_{n-1}$. Let $m_{\mathcal S_{n-1}}$ be the measure on $\mathcal S_{n-1}$ coming from the Haar measure on $\on{SL}_{n-1}(\RR)$. Bhargava-Harron prove:
\begin{theorem}[~\cite{BhaPip}, Theorem 1]\label{BhaPipMain}
	Let $n = 3,4,5$. Then:
\begin{equation}
	\frac{1}{\abs{\cup_s\mathcal F^s_n(T)}}\sum_{\FF\in \cup _{s} \mathcal F^s_n(T)}\delta_{s(\FF)}\longrightarrow m_{\mathcal S_{n-1}}\text{ weakly as }T\longrightarrow \infty.
\end{equation}
\end{theorem}
In a nutshell, we will explain why $q_E(\FF)$ contains more information than $s(\FF)$ and how to use it to deduce Bhargava-Harron's result from either one of Theorems \ref{unif3},\ref{torsion2}. Indeed, there are natural projections:
\begin{equation}
	Y_{1}^{\ell_{\overline 1}}\xrightarrow{\pi_1} X_{n-1}\xrightarrow{\pi_2} \mathcal S_{n-1}
\end{equation}
defined for any $\Lambda\in Y_1^{\ell_{\overline 1}}$ by $\pi_1(\Lambda)=\inn{P_{V_0}(\Lambda)}$ and for any $\Lambda\in X_{n-1}$ by $\pi_2(\Lambda)=\on{O}_{n-1}(\RR)\Lambda$. Recall that $q_E(\FF)$ was defined as $g^t_{E,\ell_{\overline 1}}\inn{\Lambda(\FF)}$ where $t\in \RR$ was chosen to make sure that $g^t_{E,\ell_{\overline 1}}\inn{\Lambda(\FF)}\in Y_1^{\ell_{\overline 1}}$. By definition of $g_{E,\ell_{\overline 1}}^t$, $P_{V_0}(\Lambda(\FF))$ and $P_{V_0}(g_{E,\ell_{\overline 1}}^t\inn{\Lambda(\FF)})$ have the same projection on $V_0$ up to scalar multiple. Since $s(\FF)$ was defined as the $O_{n-1}$ orbit of this projection (after normalizing by the covolume) we deduce that:
\begin{equation}\label{19:42}
	s(\FF)=\pi_2\circ\pi_1(q_E(\FF)).
\end{equation}
Let $m_{X_{n-1}}$ be the natural measure on $X_{n-1}$ coming from $\on{SL}_{n-1}(\RR)$. Denote $(\pi_1)_*m_{Y_{1}^{\ell_{\overline 1}}}$ and $(\pi_2)_*m_{X_{n-1}}$ to be the pushforward measures. It is simple to check that:
\begin{equation}\label{rivk}
	(\pi_1)_*m_{Y_{1}^{\ell_{\overline 1}}}=m_{X_{n-1}}\text{ and }	(\pi_2)_*m_{X_{n-1}}=m_{\mathcal S_{n-1}}.
\end{equation}
Moreover, the probability measures $\mu_i=\frac{1}{\mu(F_i)}\mu\mid_{F_i}$ appearing in Theorem \ref{torsion2} will turn to be invariant under the action of the group $G_0(\ell_{\overline 1}):=\{g\in \on{SL}_n(\RR):gv=v\text{ for all }v\in \ell_{\overline 1}\text{ and }gV_0=V_0\}$ and the following diagram will commute:
\begin{equation}
\begin{tikzcd}
G_0(\ell_{\overline 1}) \acts \arrow[swap]{d}{g\mapsto \phi\circ (g\mid_{V_0})\circ\phi^{-1}} & F_i \arrow{d}{\pi_1} \\
\on{SL}_{n-1}(\RR) \acts  & X_{n-1}.
\end{tikzcd}
\end{equation}
We deduce that for any $i$:
\begin{equation}\label{rivka}
	(\pi_1)_*\mu_i=m_{X_{n-1}}
\end{equation}
which together with Equation (\ref{rivk}) implies for any $i$:
\begin{equation}\label{19:41}
	(\pi_2\circ\pi_1)_*m_{Y_{1}^{\ell_{\overline 1}}}=(\pi_2\circ\pi_1)_*\mu_i=m_{\mathcal S_{n-1}}.
\end{equation}
Fix a signature $(r,s)$. By equation (\ref{19:42}) for any $T>0$: 
\begin{equation}
	\frac{1}{\abs{\mathcal F^s_n(T)}}\sum_{\FF\in \mathcal F^s_n(T)}\delta_{s(\FF)}=\frac{1}{\abs{\mathcal F^s_n(T)}}\sum_{\FF\in \mathcal F^s_n(T)}\delta_{\pi_2\circ\pi_1(q_E(\FF))}=(\pi_2\circ\pi_1)_*\left(\frac{1}{\abs{\mathcal F^s_n(T)}}\sum_{\FF\in \mathcal F^s_n(T)}\delta_{q_E(\FF)}\right)
\end{equation}
and by Theorem \ref{unif3}, equation (\ref{19:41}) and the continuity of the pushforward operation:
\begin{equation}
	\frac{1}{\abs{\mathcal F^s_n(T)}}\sum_{\FF\in \mathcal F^s_n(T)}\delta_{s(\FF)}\longrightarrow (\pi_2\circ\pi_1)_*m_{Y_{1}^{\ell_{\overline 1}}}=m_{\mathcal S_{n-1}}\text{ weakly as }T\longrightarrow\infty.
\end{equation}
Averaging the above equation over all $s=0,\dots,\floor{n/2}$ recovers Bhargava-Harron's result. Therefore Theorem \ref{unif3} implies Bhargava-Harron's Theorem \ref{BhaPipMain}. To show that Theorem \ref{torsion2} also implies Theorem \ref{BhaPipMain} note that for any $T>0$:
\begin{equation}
	\frac{1}{\abs{\mathcal F^s_n(T)}}\sum_{\FF\in \mathcal F^s_n(T)}\delta_{s(\FF)}=\frac{1}{\abs{\mathcal F^s_n(T)}}\sum_{\FF\in \mathcal F^s_n(T)}\delta_{\pi_2\circ\pi_1(q_{V_0}(\FF))}=(\pi_2\circ\pi_1)_*\left(\frac{1}{\abs{\mathcal F^s_n(T)}}\sum_{\FF\in \mathcal F^s_n(T)}\delta_{q_{V_0}(\FF)}\right)
\end{equation}
and by Theorem \ref{torsion2}, equation (\ref{19:41}) and the continuity of the pushforward operation:
\begin{equation}
	\frac{1}{\abs{\mathcal F^s_n(T)}}\sum_{\FF\in \mathcal F^s_n(T)}\delta_{s(\FF)}\longrightarrow \sum_{i=1}^ks_i\mu(F_i)(\pi_2\circ\pi_1)_*\mu_i=m_{\mathcal S_{n-1}}\text{ weakly as }T\longrightarrow\infty.
\end{equation}

\subsection{Testing the Construction on Quadratic Number Fields}
Quadratic number fields are well studied objects and their rings of integers can be fully described. Therefore, we would like to motivate our construction by examining it for quadratic number fields $\FF$. More specifically, we would like to analyze the 'asymptotic portion' of quadratic number fields with predetermined geometric behaviour phrased using $q_E$. The following theorem summerizes this analysis:
\begin{theorem}\label{n2}
Let $n=2$ and let $(r,s)$ be a signature. Let $E$ be a 1-dimensional subspace in $\RR^2$ such that $E\neq \ell_{\overline 1}, V_0$. Then:
\begin{equation}\label{63}
	\frac{1}{\abs{\mathcal F^s_2(T)}}\sum_{\FF\in \mathcal F^s_2(T)}\delta_{q_E(\FF)}\longrightarrow m_{Y_1^{\ell_{\overline 1}}}\text{ weakly as }T\longrightarrow \infty.
\end{equation}
When $E=V_0$ there exists a periodic orbit $\{p_1,p_2\}\subset Y_1^{\ell_{\overline 1}}$ and $t\in (0,1)$ such that:
 \begin{equation}\label{orrr}
	\frac{1}{\abs{\mathcal F^s_2(T)}}\sum_{\FF\in \mathcal F^s_2(T)}\delta_{q_E(\FF)}\longrightarrow t\delta_{p_1}+(1-t)\delta_{p_2}\text{ weakly as }T\longrightarrow \infty.
\end{equation}
\end{theorem}
Theorem \ref{n2} is a simple test case due to the complete description of rings of integers in quadratic fields. We use this description to sketch a proof of the theorem.

 For simplicity of this exposition, we limit ourselves to the case $s=0$. Firstly, it is easy to see that when $n=2$, $Y_1^{\ell_{\overline 1}}$ is isomorphic to the unit circle $S^1$ and that $m_{Y_1^{\ell_{\overline 1}}}$ is the Lebesgue measure $\lambda$ on $S^1$. Secondly, let $D$ be a square free natural number (nonnegative, because $s=0$). Given  $D=1\mod 4$, denote $\omega=\frac{1+\sqrt{D}}{2}$. It is well known that:
\begin{equation}
	\mathcal O_{\QQ(\sqrt{D})}=\{a+b\omega:a,b\in \ZZ\}.
\end{equation}
Let $\overline u=(1,u)\in \RR^2$ such that $u\neq -1$ and let $E=\overline{u}^{\perp}$. Using this description, $q_E(\QQ(\sqrt{D}))$ can be calculated to be (This calculation is carried in grated generality in the proof of Lemma \ref{structure} below, we refer the reader to this proof for reference):
\begin{equation}\label{oznaim}
	q_{E}(\QQ(\sqrt{D}))=\frac{1}{2(1+u)}\left(1+u+(1-u)\sqrt{D}\right)\operatorname{mod}1\in S^1.
\end{equation}
Denote for any $T>0$ and $i=0,\dots,3$:
\begin{equation}
	 \mathcal D_T(i)=\{0<D<T:\text{ }D\text{ is a square free natural number and }D=i\mod4\}.
\end{equation}
Assume further that $u\neq 1$ so that $\overline u\neq \overline 1$. Then classical asymptotic estimations on the density of square free integers (see ~\cite[Theorem 8.2.1]{IntroNumbTheory}) and equation (\ref{oznaim}) show that:
\begin{equation}\label{cin1}
	\frac{1}{\abs{\mathcal D_T(1)}}\sum_{D\in \mathcal D_T(1)}\delta_{q_{E}(\QQ(\sqrt{D}))}\longrightarrow \lambda\text{ as }T\longrightarrow\infty.
\end{equation}
When $D=2,3\mod 4$ the description of $\QQ(\sqrt{D})$ is different (and also well known), but also shows in the same way that:
\begin{equation}\label{cin2}
	\frac{1}{\abs{\mathcal D_T(2)\cup \mathcal D_T(3)}}\sum_{D\in \mathcal D_T(2)\cup \mathcal D_T(3)}\delta_{q_{E}(\QQ(\sqrt{D}))}\longrightarrow \lambda\text{ as }T\longrightarrow\infty.
\end{equation}
Averaging equations (\ref{cin1}) and (\ref{cin2}) implies the first statement of Theorem \ref{n2}.

When $\overline u=\overline 1$ and $D=1\on{mod}4$ is square free, equation (\ref{oznaim}) says $q_{V_0}(\QQ(\sqrt{D}))=q_{E}(\QQ(\sqrt{D}))=1/2+\ZZ$. When $D=2,3\on{mod}4$ the same calculation leading to equation (\ref{oznaim}) which as we said is given in Lemma \ref{structure} below, says that $q_{V_0}(\QQ(\sqrt{D}))=q_{E}(\QQ(\sqrt{D}))=\ZZ$. From this it is immediate to deduce that for some $t\in (0,1)$ (related to the density of $\mathcal D_T(1)$ in $\ZZ$):
\begin{equation}
	\frac{1}{\abs{\mathcal D_T(1)\cup \mathcal D_T(2)\cup \mathcal D_T(3)}}\sum_{D\in \mathcal D_T(1)\cup \mathcal D_T(2)\cup \mathcal D_T(3)}\delta_{q_{V_0}(\QQ(\sqrt{D}))}\longrightarrow t\delta_{0}+(1-t)\delta_{1/2}\text{ as }T\longrightarrow\infty
\end{equation}
proving equation (\ref{orrr}) and the second statement of Theorem \ref{n2}.

\newpage
\section{The Ingredients}
\subsection{Notation}\label{Not}
Given a natural number $n$, a ring of rank $n$ over another ring $R$ is a ring which is also a free $R$-module. Given a ring $R$, we denote $\operatorname{GL}_n(R)$ to be the space of all invertible $n\times n$ matrices with entries from $R$. The subset of $\operatorname{GL}_n(R)$ of all matrices with determinant $1$ is denoted $\operatorname{SL}_n(R)$. $(\RR^n,\inner{\cdot}{\cdot})$ denotes the standard inner product structure of $\RR^n$. Given a linear subspace $E\subset \RR^n$, we denote $E^{\perp}=\{y\in \RR^n:\text{ }\inner{y}{x}=0\text{ for all } x\in E\}$ and we denote $P_E:\RR^n\rightarrow E$ to be the orthogonal projection onto $E$. $\operatorname{Vol}_n$ denotes the standard Lebesgue measure on $\RR^n$. Given linear subspaces $E_1,E_2\subset \RR^n$, we write $\RR^n=E_1\oplus E_2$ if $E_1+E_2=\RR^n$ and $E_1\cap E_2=\{0\}$. The set of all matrices $g\in \operatorname{GL}_n(\RR)$ such that $g^tg=gg^t=I$ is denoted by $O(n)$, and $\operatorname{O}(n)\cap \operatorname{SL}_n(\RR)$ is denoted $\operatorname{SO}(n)$. For a natural number $n$ we denote $\Sigma_n$ to be a fundamental domain for the action of $\operatorname{SL}_{n}(\ZZ)$ on $\operatorname{SL}_{n}(\RR)$. When additional properties of the fundamental domain are required, we will state so explicitly. Given a locally compact topological space $X$ and Borel measures $(\mu_n)_{n\geq 0},\mu$ on $X$, we say that $\mu$ is a {weak limit} of $\mu_n$ if for any continuous and compactly supported function $f:X\rightarrow \RR$, $\int fd\mu_n\longrightarrow \int fd\mu$. Given a finite Borel measure $\mu$ on a topological space $X$, a continuous map $\pi:X\rightarrow Y$ to another topological space $Y$, $\pi_*\mu$ denotes the push-forward measure of $\mu$, which is a Borel measure on $Y$, defined by $\pi_*(\mu)(A)=\mu(\pi^{-1}A)$ for any Borel set $A\subset Y$. In case $\mu$ is any measure (not necessarily finite), a Jordan measurable subset $A\subset X$ is a Borel measurable set such that $\mu(\partial A)=0$. Given a subset $R\subset \RR^m$, we denote $N(R)$ to be the size of $R\cap \ZZ^m$. Given $n\in \NN$, we denote $[n]=\{1,\dots,n\}.$ A number field is a finite dimensional field extension $\FF\supset \QQ$. Given a number field $\FF$ of degree $n$ and number $s=1,\dots,\floor{\frac{n}{2}}$, we say that $\FF$ is {of signature $(n-2s,s)$} if the minimal polynomial $p$ of $\FF$ has precisely $s$ pairs of conjugate complex roots. Given a natural numbers $n,k,d$, a function $f:\RR^n\rightarrow \RR^k$ is said to be {homogeneous of degree $d$} if $f(\lambda v)=\lambda^df(v)$ for every $v\in \RR^n$ and $\lambda>0$. An $n$-dimensional lattice is a subset of $\RR^n$ given by $\ZZ$-span of a basis for $\RR^n$. The {covolume} of a lattice in a Euclidean space $\Lambda\subset (\RR^n,\inner{\cdot}{\cdot})$ denoted $\cov{\Lambda}$, is defined as $\operatorname{Vol_n}(\mathcal F)$ and $\mathcal F$ is any fundamental domain for the additive action of $\Lambda$ on $\RR^n$. Given a lattice $\Lambda\subset \RR^n$, $\Lambda^*$ is the dual lattice defined by $\{y\in \RR^n:\inner{x}{y}\in \ZZ\text{ for all }x\in \Lambda\}$.
 Given an $n$-dimensional lattice $\Lambda$, we denote $\inn{\Lambda}=\cov{\Lambda}^{-1/n}\Lambda$. $S_n$ denotes the symmetric group of order $n$. The space of unimodular lattices in $\RR^n$ is denoted $X_n$ and sometimes $X(\RR^n)$.

\newpage
\subsection{New Lattice Normalization}\label{NLN}
In this subsection we repeat in more detail the construction carried in the introduction.\\
For every degree $n$ number field $\FF$ of signature $(r,s)$ we denote $\sigma_1,\dots,\sigma_r$ to be a fixed ordering of the real embeddings and $\tau_1,\overline{\tau_1},\dots,\tau_s,\overline {\tau_s}$ to be a fixed ordering of the pairs of complex embeddings. We will use this ordering throughout the paper. To stress this fact, we introduce:
 \begin{nota}\label{ord}
 	For any degree $n$ number field $\FF$ of signature $(r,s)$, $\{\sigma_1,\dots,\sigma_r,\tau_1,\overline{\tau_1},\dots,\tau_s,\overline {\tau_s}\}$ denotes a fixed arbitrary ordering of the natural embeddings of $\FF$.
\end{nota}
For every $\pi=(\pi_r,\pi_s)\in S_r\times S_s$, we denote $\Lambda_{\pi}(\FF)$ to be the image of $\mathcal O_{\FF}$ under
\begin{equation}
	(\sigma_{\pi_r(1)},\dots,\sigma_{\pi_r(r)},\tau_{\pi_s(1)},\dots,\tau_{\pi_s(s)}):\FF\rightarrow \RR^r\oplus\CC^s\cong\RR^n.
\end{equation}
It is well known that $\Lambda_{\pi}(\FF)$ is a lattice in $\RR^n$. Define the vector $\overline 1\vcentcolon=(1,\dots,1,1,0,\dots,1,0)\in \RR^n$ to be the image of $(1_{\RR},\dots,1_{\RR},1_{\CC},\dots,1_{\CC})\in \RR^r\oplus \CC^s$ under the identification with $\RR^n$. Let $\inner{\cdot}{\cdot}$ be the standard inner product on $\RR^n$, let $V_0=\{x\in \RR^n:\inner{x}{\overline 1}=0\}$ and denote $\ell_{\overline 1}=\RR\cdot \overline 1$. Now we carry in detail the construction of $q_E(\FF)$ from the introduction.

Since the language of grids will be useful for us, we introduce the following definition:
\begin{definition}
	The space of $n$-dimensional grids is defined by:
	\begin{equation}
		Y_n=\{\Lambda+v:\Lambda\in X_n,v\in \RR^n\}.
	\end{equation}
Any $M\in Y_n$ uniquely determines a lattice $\Lambda_M\in X_n$ such that there exists $v\in \RR^n$, defined up to $\Lambda_M$, satisfying $M=\Lambda_M+v$. Denote the corresponding element in $\RR^n/\Lambda$ by $\text{vec}(M)$ and call it the translating vector of $M$. 
\end{definition}
Note that $X(\RR^n)$ is embedded naturally inside $Y_n$ by the identity map (taking $v=0$ in the definition of the grid). Next, we come back to the definition of the space in which we shall have equidistribution.
\begin{definition}
For any line $\ell$ in $\RR^n$ define:
\begin{equation*}
	Y^{\ell}(\RR^n)\vcentcolon=\{\Lambda\in X_n:\Lambda\cap \ell\text{ is a lattice in }\ell\}
\end{equation*}
and
\begin{equation*}
	Y_1^{\ell}(\RR^n)\vcentcolon=\{\Lambda\in X_n:\Lambda\cap \ell\text{ is a unimodular lattice in }\ell\}
\end{equation*}
\end{definition}
The relation between this definition and the language of grids is the following. For any line $\ell$ in $\RR^{n+1}$ there is a natural bijection from $Y_1^{\ell}(\RR^{n+1})$ to $Y_n$. Indeed, for any $\Lambda\in Y_1^{\ell}(\RR^{n+1})$ there exists $u(\ell,\Lambda)\notin \ell^{\perp}$ such that:
\begin{equation}\label{14:19}
	\Lambda^*=\bigcup_{k\in \ZZ} \Lambda^*\cap \ell^{\perp}+ku(\ell,\Lambda).
\end{equation}
The vector $u(\ell,\Lambda)$ is unique up to choice of direction for $\ell$. Fix such choice and define:
\begin{definition}\label{rdef}
	The function $\rho^{\ell}_{n+1}:Y_1^{\ell}(\RR^{n+1})\rightarrow Y_n$ is defined by:
	\begin{equation}
		Y_1^{\ell}(\RR^{n+1})\ni\Lambda \mapsto (\Lambda^*\cap \ell^{\perp})+P_{\ell^{\perp}}(u(\ell,\Lambda))
	\end{equation}
	where $u(\ell,\Lambda)\in \RR^{n+1}$ is as above.
\end{definition}
We will not prove the following lemma:
\begin{Lemma}
	For any $n$ and $\ell$ as above, the map $\rho_{n}^{\ell}$ is a bijection. 
\end{Lemma}
The submanifold $Y_1^{\ell}(\RR^{n+1})$ is a periodic orbit of the group:
\begin{equation*}
	G(\ell)\vcentcolon=\{g\in \operatorname{SL}_{n+1}(\RR):g\mid_{\ell}=\on{Id}\mid_{\ell}\}\cong\operatorname{SL}_n(\RR)\ltimes\RR^n\leq \on{SL}_{n+1}(\RR).
\end{equation*}
This group also acts on $Y_n$ in a natural transitive way making the map $\rho_{n+1}^{\ell}$ equivariant. Denote $m_{Y_1^{\ell}(\RR^{n+1})}$ to be the Haar measure on $Y_1^{\ell}(\RR^{n+1})$ coming from periodicity of the $G(\ell)$-orbit $Y_1^{\ell}(\RR^{n+1})$. The measure $m_{Y_1^{\ell}(\RR^{n+1})}$ and the pushforward measure $m_{Y_n}\vcentcolon=(\rho^{\ell}_{n+1})_*m_{Y_1^{\ell}(\RR^{n+1})}$ are thus invariant under the action of $G(\ell)$. 

Now that we finished describing the probability space where equidistribution will occur, we discuss the way to deform our lattices coming from number fields into it. As in the introduction, the tool for this purpose is the following family of one parameter groups:
\begin{definition}\label{norm}
	Given subspaces $E_1,E_2\subset \RR^n$ such that $E_1\oplus E_2=\RR^n$ and $t\in \RR$, we define a linear operator $g_{E_1,E_2}^t\in \operatorname{SL}_n(\RR)$ by $g(x)=e^{t/dim(E_1)}x$, $g(y)=e^{-t/dim(E_2)}y$ for every $x\in E_1,y\in E_2$. The set $\{g_{E_1,E_2}^t\}_{t\in \RR}$ is clearly a group, which we shall call the {$E_1,E_2$ normalization group}.
\end{definition}
Next we explain how to carry the deformation from $Y^{\ell_{\overline 1}}$ into $Y^{\ell_{\overline 1}}_1$ using the normalization groups. In particular, we do that for the lattices coming from number fields.
\begin{definition}\label{notations}
	Let $\Lambda\in Y^{\ell_{\overline 1}}$. Let $E\subset \RR^n$ be such that $\ell_{\overline 1}\oplus E=\RR^n$. Let $t\in \RR$ be such that $e^{-t}\operatorname{cov}(\Lambda)^{-1/n}\sqrt{n}=1$. We define:
	\begin{equation}
		\Lambda(E)=g_{E,\ell_{ \overline 1}}^t(\Lambda). 
	\end{equation}
For every degree $n$ number field $\FF$ of signature $(r,s)$ and $\pi=(\pi_r,\pi_s)\in S_r\times S_s$ we define:
\begin{equation}\label{20:30}
	\Lambda^{\pi}_{\FF}(E)=(\Lambda_{\pi}(\FF))(E).
\end{equation}
Note that by the choice of $t$ it holds that $\Lambda(E)\in Y_1^{\ell_{\overline 1}}$. Thus $\Lambda(E)\in Y_1^{\ell_{\overline 1}}(\RR^n)$ and in particular so is $\Lambda^{\pi}_{\FF}(E)$. Denote:
\begin{equation}
	\Gamma_E(\Lambda)=\rho_n^{\ell_{\overline 1}}(\Lambda(E))
\end{equation}
and
\begin{equation}
	\Gamma_E^{\pi}(\FF)=\Gamma_E(\Lambda^{\pi}_{\FF}(E)).
\end{equation}
\end{definition}
We restate Theorem \ref{unif3} from the introduction in the language of grids, because this will be more natural for our proof later.
\begin{theorem}\label{unif3g}
Let $n = 3,4,5$ and let $(r,s)$ be a signature. Let $E$ be a subspace of $\RR^n$ of such that $E\oplus \ell_{\overline 1}=\RR^n$ and $E\neq V_0$. Then:
\begin{equation}
	\frac{1}{\abs{\mathcal F^s_n(T)}}\sum_{\pi\in S_r\times S_s}\sum_{\FF\in \mathcal F^s_n(T)}\delta_{\Gamma_E^{\pi}(\FF)}\longrightarrow m_{Y_{n-1}}\text{ weakly as }T\longrightarrow \infty.
\end{equation}
\end{theorem}
Next we deal with the case $E=V_0$. To state our next result, we need the following definition:
\begin{definition}\label{torDef}
Let $\Lambda\in X_n,m\in \NN$. Define:
\begin{equation}
	Y_n(m)=\{\Lambda+v\in Y_n:mv\in \Lambda,m'v\notin \Lambda\text{ for all }m'<m\}.
\end{equation}
Note that $Y_n(m)$ is identified under $\rho_n^{\ell_{\overline 1}}$ with a periodic orbit of $G(\ell_{\overline 1})$ in $X_{n}$ and thus exists unique Borel probability measure $m_{Y_n(m)}$ supported on $Y_n(m)$ and invariant under this action.
\end{definition}
Now we can deal with the case $E=V_0$.
\begin{theorem}\label{torsions2g}
	Let $n=3,4,5$ and let $(n-2s,s)$ be a signature. Denote $\on{Div}(n)=\{1\leq k\leq n:k\mid n\}$ to be the set of divisors of $n$. Then for any $k\in \on{Div}(n)$ there exists $\ell=\ell(k)\in \NN$, disjoint closed connected subsets $F^{(k)}_1,\dots,F^{(k)}_{\ell(k)}\subset Y_n(k)$ and nonnegative scalars $s_{k,\ell};\ell=1,\dots,\ell(k)$ not all zero, such that:
\begin{equation}
	\frac{1}{\abs{\mathcal F^s_n(T)}}\sum_{\pi\in S_r\times S_s}\sum_{\FF\in \mathcal F^s_n(T)}\delta_{q_{V_0}(\FF)}\longrightarrow \sum_{k\in \on{Div}(k)}\sum_{l=1}^{\ell(k)}s_{k,\ell}m_{Y_n(k)}\mid_{F_{\ell}^{(k)}}\text{ weakly as }T\longrightarrow \infty.
\end{equation}
\end{theorem}
\newpage
\subsection{Basic Subsets}
\begin{definition}\label{BS}
	Given a lattice $\Lambda=\operatorname{sp}_{\ZZ}\{w_1,\dots,w_n\}\subset \RR^n$, a fundamental domain $\Sigma$ for the action of $\operatorname{stab}_{\operatorname{SL}_n(\RR)}(\Lambda)$ on $\operatorname{SL}_n(\RR)$, an open and bounded subset $S\subset \Sigma$ and $U=I_1e_1+\dots+I_ne_n\subset [0,1]^n$ where $I_i\subset [0,1]$ is an interval (such subset will be denoted by {box}), we define:
	\begin{equation*}
		S\times_{\Sigma}U=\{g^{-1}(\Lambda+I_1w_1+\dots+I_nw_n):g\in S\}\subset Y_n.
	\end{equation*}
	We refer to subsets of the above form as $\Lambda$-$\Sigma$-basic subsets.
\end{definition}
\begin{definition}\label{UW}
	Given $U$ as in Definition \ref{BS} and an ordered basis $w=\{w_1,\dots,w_n\}$ of $\RR^n$, denote:
	\begin{equation*}
		U_w=I_1w_1+\dots+I_nw_n.
	\end{equation*}
\end{definition}
The following technical Lemma will not be proved.
\begin{Lemma}\label{Bas}
Using the notation of Definition \ref{BS}, the collection $\mathcal T=\{S\times_{\Sigma}U:S\subset \Sigma,\text{such that }\sigma(\partial S)=0\text{, }U\subset [0,1]^n\text{ a box}\}$ consists of subsets of $m_{Y_{n}}$-measure-$0$ boundary which constitute a basis for the natural topology on $Y_n$. 
\end{Lemma}
\subsection{Bhargava's Correspondence}
Some of the main tools we use (as do Bhargava-Harron in ~\cite{BhaPip}) involve parametrizations all cubic, quartic and quintic orders which are carried in ~\cite{Bha1,Bha2,FadBook}. 
\begin{definition}
	Let $T$ be a ring. Define $V_T$ to be:
	\begin{enumerate}
		\item the space $Sym^3 T^2(\otimes T)$ of binary cubic forms over $T$, if $n = 3$;
		\item the space $Sym^2 T^3\otimes T^2$ of pairs of ternary quadratic forms over $T$, if $n = 4$;
		\item the space $T^4\otimes\wedge^2 T^5$ of quadruples of alternating quinary $2$-forms over $T$ , if $n = 5$.
	\end{enumerate}
For $n=3,4,5$, we set $r=r(n)=2,3,6$ respectively. Note that $G_T=\operatorname{GL}_{n-1}(T)\times \operatorname{GL}_{r-1}(T)$ acts naturally on $V_T$. The discriminant of an element $v\in V_T$ is a polynomial of degree $d$ in the coefficients of $v$, where $d=4,12,40$ when $n=3,4,5$ respectively (see ~\cite{DiscPaper}). 
\end{definition}
As explained in ~\cite{BhaPip}, the following Theorem can be deduced from ~\cite[\S15]{FadBook}, ~\cite[Corollary 5]{Bha1} and ~\cite[Corollary 3]{Bha2} and is Theorem 2 in ~\cite{BhaPip}:
\begin{theorem}\label{par1}
	The nondegenerate (i.e. with non-zero discriminant) elements of $V_{\ZZ}$ are in canonical bijection with isomorphism classes of pairs $((R,\alpha),(S,\beta))$, where $R$ is a nondegenerate ring of rank $n$ and $S$ is a rank $r$ resolvent ring of $R$, and $\alpha$ and $\beta$ are $\ZZ$-bases for $R/\ZZ$ and $S/\ZZ$, respectively. In this bijection, the discriminant of an element of $V_{\ZZ}$ is equal to the discriminant of the corresponding ring $R$ of rank $n$. Moreover, under this bijection, the action of $G_{\ZZ}$ on $V_{\ZZ}$ results in the corresponding natural action of $G_{\ZZ} = GL_{n-1}(\ZZ)\times GL_{r-1}(\ZZ)$ on $(\alpha, \beta)$. Finally, every isomorphism class of maximal ring $R$ of rank $n$ arises in this bijection, and the elements of $V_{\ZZ}$ yielding $R$ consists of exactly one $G_{\ZZ}$-orbit. We denote the rings corresponding to $v\in V_{\RR}$ by $R_{\ZZ}(v),S_{\ZZ}(v)$ and the bases corresponding to $v$ by $\alpha(v),\beta(v)$.
\end{theorem}
A {resolvent ring} of a cubic, quartic or quintic ring is some quadratic, cubic, or sextic ring respectively, which satisfies some conditions (see ~\cite{Bha1,Bha2} for more details). Since the precise definition is not required here, we choose not to get into the details of this matter. Theorem \ref{par1} holds with any field $K$ in place of $\ZZ$ with the same proofs as in ~\cite{Bha1,Bha2}. In particular, taking $K=\RR$:
\begin{theorem}\label{par2}
	 There is a canonical bijection between the nondegenerate elements of $V_{\RR}$ and isomorphism classes of pairs $((R,\alpha),(S,\beta))$, where $R$ is a nondegenerate ring of rank $n$ over $R$ and $S$ is the (unique) rank $r$ resolvent ring of $R$ over $\RR$, and $\alpha$ and $\beta$ are $\RR$-bases for $R/\RR$ and $S/\RR$, respectively. Moreover, under this bijection, the action of $G_{\RR}$ on $V_{\RR}$ results in the corresponding natural action of $G_{\RR}$ on $(\alpha, \beta)$. We denote the rings corresponding to $v\in V_{\RR}$ by $R(v),S(v)$.
\end{theorem}
\begin{remark}\label{bhaComp}
Theorems \ref{par1},\ref{par2} are compatible in the following sense. On the one hand $V_{\ZZ}\subset V_{\RR}$ naturally and if $v\in V_{\ZZ}$ corresponds under Theorem \ref{par1} to the ring $R$, then it corresponds under Theorem \ref{par2} to $R\otimes \RR$. Moreover, given $v\in V_{\ZZ}$ the multiplication tables of the algebras $R,S$ corresponding to $v$ under Theorem \ref{par1} with respect to the bases $\alpha,\beta$ are the same as the multiplication tables of $R',S'$ with respect to the bases $\alpha',\beta'$ which correspond to $v$ when thought of as an element of $V_{\RR}$. This follows from ~\cite[\S15 (1) and (2)]{FadBook} when $n=3$, ~\cite[(14),(21),(22) and (23)]{Bha1} when $n=4$ and by ~\cite[(16),(17),(21) and (22)]{Bha2} when $n=5$, and is written also in the top of page 5 of ~\cite{BhaPip}.	
\end{remark}
\begin{remark}\label{orbs}
	Rank-$n$ rings $R$ over $\RR$ are in particular {\'Etale algebras over $\RR$} of rank $n$, which implies $R\cong \RR^{n-2k}\times \CC^k$ for some $k=0,\dots,\floor{n/2}$. Following the notation of ~\cite{BhaPip}, given $k=0,\dots,\floor{n/2}$, we denote $V_{\RR}^{(k)}$ to be the subset of $V_{\RR}$ of elements for which the corresponding ring $R$ (recall Theorem \ref{par2}) has the structure $\RR^{n-2k}\times \CC^k$. Therefore, and as mentioned in ~\cite{BhaPip}, the nondegenerate orbits for the action of $G_{\RR}$ on $V_{\RR}$ are precisely $V_{\RR}^{(0)},\dots,V_{\RR}^{(\floor{n/2})}$. Moreover, recalling Remark \ref{bhaComp}, we denote $V_{\ZZ}^{(i)}=V_{\RR}^{(i)}\cap V_{\ZZ}$ for $i=1,\dots,\floor{n/2}$.
\end{remark}
\begin{definition}\label{irma}
	Given $v\in V_{\ZZ}$, denote by $R(v)$ the ring corresponding to it under Theorem \ref{par2}.
	\begin{enumerate}
		\item We say that $v$ is {irreducible} if $R(v)$ is isomorphic to an order inside a type $S_n$ number field;
		\item We say that $v$ is {maximal} if $R(v)$ is a maximal ring of rank $n$ over $\ZZ$ (which means it cannot be embedded as a sub $\ZZ$ module inside any other ring of rank $n$ over $\ZZ$).
	\end{enumerate}
	Recalling the natural identification of $V_{\ZZ}$ inside $\ZZ^d$ for an appropriate $d$ (depending on $n=3,4,5$), we say that $v\in \ZZ^d$ is irreducible (respectively, maximal) if the corresponding element of $V_{\ZZ}$ is irreducible (respectively, maximal).
\end{definition}
\begin{Corollary}\label{1:1}
	Let $n=3,4,5$ and $X>0$. There is $1:1$ correspondence between $G_{\ZZ}$-orbits of irreducible maximal elements of $V_{\ZZ}^{(s)}$ with discriminant less than $X$ and rings of integers of isomorphism classes of signature $(r,s)$ number fields with discriminant less than $X$. The correspondence is given by $G_{\ZZ}.v\leftrightarrow R_{\ZZ}(v)$.
\end{Corollary}
\begin{definition}\label{contBaseDef}
	Given $n=3,4,5$ we let $v_0\in V_{\RR}$ and let $u^{\perp}=E\subset \RR^n$ such that $u\notin V_0$.
	\begin{enumerate}
		\item For any $v\in V_{\RR}$ let $\operatorname{MT}_v$ denote the multiplication table \ref{MT3} when $n=3$, \ref{MT4} when $n=4$ and \ref{MT5} when $n=5$. Given a rank $n$ ring $R$ and $x_1,\dots,x_{n-1}\in R$ we write $\operatorname{MT}_v(x_1,\dots,x_{n-1})$ if $(x_i)_{i=1}^{n-1}$ satisfies $\operatorname{MT}_v$;
		\item For $i=1,\dots, n-1$, $\overline \alpha_i:V_{\RR}\rightarrow \RR^n$ denote the continuous vector functions such that:
		\begin{equation}
			\operatorname{MT}_{v}(\overline \alpha_1(v),\dots,\overline \alpha_{n-1}(v))
		\end{equation}
	for any $v\in V_{\RR}$ whose existence is guaranteed by \ref{basCons};
		\item Given $g\in \RR_+\Sigma_{n-1}\times \Sigma_{r-1}$ denote $\Lambda_{v_0}(g)=\operatorname{span}_{\ZZ}\{\overline 1,\overline \alpha_1(gv_0),\dots,\overline \alpha_n(gv_0)\}$;
		\item Given a basis $w=(w_j)_{j=1}^{n-1}$ of $V_0$, we define a function $f_{w}:\RR_+\cdot(\Sigma_{n-1}\times\Sigma_{r-1})v_0\rightarrow V_0$ by:
	\begin{equation*}
		f_w(p)=\frac{1}{\inner{u}{1}}\sum_{i=1}^{n-1}\inner{u}{\overline{\alpha}_i(g_pv_0)}w_i,
	\end{equation*}
	where for $p\in \RR_+\cdot\Sigma_{n-1} v_0$, $g_p$ is the element of $\RR_+\cdot\Sigma_{n-1}\times \Sigma_{r-1}$ satisfying $g_pv_0=p$;
	\item Given a basis $w=\{w_1,\dots,w_{n-1}\}$ of $V_0$, we denote:
		\begin{equation}
			\Lambda_w=\operatorname{span}_{\ZZ}\{w_1,\dots,w_{n-1}\}.
		\end{equation}
	\end{enumerate}
	\end{definition}
The main goal of this Subsection is to prove the following structural Lemma:
\begin{Lemma}\label{structure}
	Let $n=3,4,5$, $i=1,\dots,\floor{n/2}$ and let $u^{\perp}=E\subset \RR^n$ be a subspace of co-dimension 1 not containing $\overline 1$. Let $\Sigma_{n-1}'$ be any fundamental domain for the action of $\operatorname{SL}_{n-1}(\ZZ)$ on $\operatorname{SL}_{n-1}(\RR)$ and let $v_0\in V_{\RR}^{(i)}$. Then there exists a fixed basis of $V_0$ given by $w=(w_j)_{j=1}^{n-1}$, a function $\Pi:\mathcal F_{\infty}^{(i)}\rightarrow S_n$ (recall Subsection \ref{Not}) and $g_0\in \operatorname{SL}_{n-1}(\RR)$ such that the fundamental domain:
	\begin{equation}
		\Sigma_{n-1}\vcentcolon=g_0^t\Sigma_{n-1}'g_0^{-t}
	\end{equation}
satisfies that for any $\Lambda_w$-$\Sigma_{n-1}$-basic subset $S\times_{\Sigma_{n-1}} U\subset Y_{n-1}$:
	\begin{equation}
		\mathcal F^{(i)}_X\cap (\Gamma_E^{\Pi})^{-1}(S\times_{\Sigma_{n-1}} U)\xleftrightarrow{\sim}
	\end{equation}
	\begin{equation*}
		\{\text{irreducible, maximal points inside }[0,X] (g_0^{-t}S^{-1}g_0^{t}\times \Sigma_{r-1})v_0\}\cap ((f_w+q_0)\operatorname{mod}\Lambda_{w})^{-1}(U_w)
	\end{equation*}
	where $q_0\in V_0$ is some fixed vector (recall Definition \ref{UW} for definition of $U_w$).  Note that the function $\Gamma_E^{\Pi}:\mathcal F_{\infty}^{(i)}\rightarrow Y_{n-1}$ is defined by $\mathcal F_{\infty}^{(i)}\ni\FF\mapsto \Gamma_E^{\Pi(\FF)}(\FF)$. Moreover, for any $\pi\in S_n$ the function $\Pi$ can be replaced with $\pi\circ\Pi$ (where $(\pi\circ\Pi)(\FF)\vcentcolon=\pi\circ (\Pi(\FF))$).
\end{Lemma}
\begin{proof}
	Let $i=1,\dots,\floor{n/2}$ and pick $v_0\in V_{\RR}^{(i)}$. Pick any fundamental domain $\Sigma_{n-1}'$ for the action of $\operatorname{SL}_{n-1}(\ZZ)$ on $\operatorname{SL}_{n-1}(\RR)$ and similarly pick $\Sigma_{r-1}$ to be a fundamental domain for the action of $\operatorname{SL}_{r-1}(\ZZ)$ on $\operatorname{SL}_{r-1}(\RR)$. For any $v=t(g,h)v_0\in \RR_+\cdot (\Sigma'_{n-1},\Sigma_{r-1})v_0$ recall Definiton \ref{contBaseDef} and denote:
\begin{equation}
	\Lambda_v=\Lambda_{v_0}(t(g,h)).
\end{equation}
Next, write:
\begin{equation}
	\Lambda_v=\operatorname{span}_{\ZZ}\{\overline{\alpha}_1(v),\dots,\overline{\alpha}_{n-1}(v),\overline 1\},D_v=\operatorname{cov}(\Lambda_v)^{-1/n}.
\end{equation}	
This is the lattice we start with. Next we follow the steps of normalization which are described in Equation (\ref{20:30}) and find $\Lambda_v(E)$ and $\Gamma_E(\Lambda(v))$ explicitly using $v_0,t(g,h)$. The first step is to normalize the covolume of $\Lambda_v$ by defining:
\begin{equation}
	\Lambda_v^0=\operatorname{span}_{\ZZ}\{D_v\overline{\alpha}_1(v),\dots,D_v\overline{\alpha}_{n-1}(v),D_v\overline{1}\}.
\end{equation}
The next step is acting on $\Lambda^0_v$ with the $E,\RR\cdot \overline{1}$ normalization group from Definition \ref{norm}. To this end, write: 
\begin{equation}
	\overline{\alpha}_i=(\overline{\alpha_i})_{\perp}+c_i\overline 1\text{ such that }(\overline{\alpha}_i)_{\perp}\in E,i=1,\dots,n-1
\end{equation}
and obviously:
\begin{equation}
	D_v\overline{\alpha}_i=D_v(\overline{\alpha_i})_{\perp}+D_vc_i\overline 1\text{ such that }(\overline{\alpha}_i)_{\perp}\in E,i=1,\dots,n-1.
\end{equation}
Denote $\lambda_v=n^{-1/2}D_v^{-1}$. The appropriate element $g_t$ from the $E,\RR\overline 1$ normalization group is defined uniquely by:
\begin{equation*}
	g_t(\overline 1)=\lambda_v\overline 1;
\end{equation*}
\begin{equation*}
	g_t(e)=\lambda_v^{-\frac{1}{n-1}}e\text{ for every }e\in E.
\end{equation*}
Denote also:
\begin{equation}
	\alpha_i^0\vcentcolon=\lambda_v^{-\frac{1}{n-1}}D_v(\overline{\alpha}_i)_{\perp}+c_iD_v\lambda_v\overline 1=g_t(D_v\overline{\alpha}_i).
\end{equation}
Let
\begin{equation}\label{brei}
	\beta_i=\frac{\lambda_v^{\frac{1}{n-1}}\left((\overline{\alpha}_1)_{\perp}\times\dots \times(\overline{\alpha}_{i-1})_{\perp}\times\overline 1\times(\overline{\alpha}_{i+1})_{\perp}\dots \times (\overline{\alpha}_{n-1})_{\perp}\right)}{D_v\inner{(\overline{\alpha}_1)_{\perp}\times\dots\times(\overline{\alpha}_{n-1})_{\perp}}{\overline 1}}
\end{equation}
and denote:
\begin{equation}\label{tling}
		\tau_v\vcentcolon=\frac{\alpha_1^0\times\dots\times\alpha_{n-1}^0}{\inner{\alpha_1^0\times\dots\times\alpha_{n-1}^0}{D_v\lambda_v\overline 1}},P_{V_0}(\tau_v)= \frac{P_{V_0}(\alpha_1^0\times\dots\times\alpha_{n-1}^0)}{\inner{\alpha_1^0\times\dots\times\alpha_{n-1}^0}{D_v\lambda_v\overline 1}}=\frac{P_{V_0}(u)}{\inner{u}{1}}+\sum_{i=1}^{n-1}c_i\beta_i.
	\end{equation}
Note that we have found $\Lambda_v(E)$ and it is given by:
\begin{equation}\label{15:04}
	\Lambda_v(E)=\operatorname{span}_{\ZZ}\{\lambda_vD_v\overline 1,\alpha_1^0,\dots,\alpha_{n-1}^0\}
\end{equation}
(recall Definition \ref{notations}). Next we find $\Gamma_E(\Lambda_v)$. By Equation (\ref{14:49}): 
\begin{equation}\label{2:05}
	P_{V_0}\Lambda_v(E)=g_0g^tg_0^{-1}(P_{V_0}\Lambda_{v_0}(E))
\end{equation}
where:
\begin{equation*}
	g_0=\begin{pmatrix}
	\mid\ & \cdots &\mid \\
	P_{V_0}\alpha_1^0(v_0) & \cdots & P_{V_0} \alpha_{n-1}^0(v_0)\\
	\mid & \cdots & \mid
	\end{pmatrix}
\end{equation*}
and therefore $\beta_i=\beta_i(g)$ depends only on $g$ for every $i=1,\dots,n-1$. We claim that $\{\beta_i(g)\}_{i=1}^{n-1}$ is a basis for the lattice:
	\begin{equation}\label{2:06}
		\left(\Lambda_{v}(E)\right)^*\cap V_0=\left(P_{V_0}\Lambda_v(E)\right)^*=\left(g_0g^tg_0^{-1}(P_{V_0}\Lambda_{v_0}(E))\right)^*
	\end{equation}
where the first equality holds by known identities of the dual lattice and the second by Equation (\ref{2:05}). Indeed, by definition of the dual lattice, it suffices to check that $\inner{\beta_i}{\alpha_j^0}\in\{\delta_{ij},-\delta_{ij}\}$ and that $\inner{\beta_i}{\lambda_vD_v\overline 1}=0$ for every $i,j=1,\dots,n-1$. Indeed, let $i=1,\dots,n-1$. Then:
\begin{equation}
	\inner{\beta_i}{\alpha_i^0}=\inner{\frac{\lambda_v^{\frac{1}{n-1}}\left((\overline{\alpha}_1)_{\perp}\times\dots \times(\overline{\alpha}_{i-1})_{\perp}\times\overline 1\times(\overline{\alpha}_{i+1})_{\perp}\dots \times (\overline{\alpha}_{n-1})_{\perp}\right)}{D_v\inner{(\overline{\alpha}_1)_{\perp}\times\dots\times(\overline{\alpha}_{n-1})_{\perp}}{\overline 1}}}{\lambda_v^{-\frac{1}{n-1}}D_v(\overline{\alpha}_i)_{\perp}+c_iD_v\lambda_v\overline 1}
\end{equation}
\begin{equation*}
	=\inner{\frac{\lambda_v^{\frac{1}{n-1}}\left((\overline{\alpha}_1)_{\perp}\times\dots \times(\overline{\alpha}_{i-1})_{\perp}\times\overline 1\times(\overline{\alpha}_{i+1})_{\perp}\dots \times (\overline{\alpha}_{n-1})_{\perp}\right)}{D_v\inner{(\overline{\alpha}_1)_{\perp}\times\dots\times(\overline{\alpha}_{n-1})_{\perp}}{\overline 1}}}{\lambda_v^{-\frac{1}{n-1}}D_v(\overline{\alpha}_i)_{\perp}}
\end{equation*}
\begin{equation*}
	=\inner{\frac{\left((\overline{\alpha}_1)_{\perp}\times\dots \times(\overline{\alpha}_{i-1})_{\perp}\times\overline 1\times(\overline{\alpha}_{i+1})_{\perp}\dots \times (\overline{\alpha}_{n-1})_{\perp}\right)}{\inner{(\overline{\alpha}_1)_{\perp}\times\dots\times(\overline{\alpha}_{n-1})_{\perp}}{\overline 1}}}{(\overline{\alpha}_i)_{\perp}}
\end{equation*}
\begin{equation*}
	=\frac{\inner{(\overline{\alpha}_1)_{\perp}\times\dots \times(\overline{\alpha}_{i-1})_{\perp}\times\overline 1\times(\overline{\alpha}_{i+1})_{\perp}\dots \times (\overline{\alpha}_{n-1})_{\perp}}{(\overline{\alpha}_i)_{\perp}}}{\inner{(\overline{\alpha}_1)_{\perp}\times\dots\times(\overline{\alpha}_{n-1})_{\perp}}{\overline 1}}
\end{equation*}
\begin{equation*}
	=\frac{(-1)^i\inner{(\overline{\alpha}_1)_{\perp}\times\dots\times(\overline{\alpha}_{n-1})_{\perp}}{\overline 1}}{\inner{(\overline{\alpha}_1)_{\perp}\times\dots\times(\overline{\alpha}_{n-1})_{\perp}}{\overline 1}}=(-1)^i
\end{equation*}
where the first equality holds by definition of $\beta_i,\alpha_i^0$, the second holds by definition of the cross product, the third is elimination of the constants, the forth is taking out the constant and the fifth holds by two general properties of the mixed product saying that for any $u_1\dots,u_n\in \RR^n$:
\begin{equation*}
	\inner{u_1\times\dots\times u_{n-1}}{u_n}=\inner{u_n\times u_1\dots\times u_{n-2}}{u_{n-1}}.
\end{equation*}
and that for any $\pi\in S_{n-1}$:
\begin{equation*}
	u_1\times\dots\times u_{n-1}=\operatorname{sgn}(\pi)u_{\pi(1)}\times\dots\times u_{\pi(n-1)}.
\end{equation*}
Let $i\neq j$ be in $\{1,\dots,n-1\}$. Then:
\begin{equation}
	\inner{\beta_i}{\alpha_j^0}=\inner{\frac{\lambda_v^{\frac{1}{n-1}}\left((\overline{\alpha}_1)_{\perp}\times\dots \times(\overline{\alpha}_{i-1})_{\perp}\times\overline 1\times(\overline{\alpha}_{i+1})_{\perp}\dots \times (\overline{\alpha}_{n-1})_{\perp}\right)}{D_v\inner{(\overline{\alpha}_1)_{\perp}\times\dots\times(\overline{\alpha}_{n-1})_{\perp}}{\overline 1}}}{\lambda_v^{-\frac{1}{n-1}}D_v(\overline{\alpha}_j)_{\perp}+c_jD_v\lambda_v\overline 1}=0
\end{equation}
because the element in the left side of the inner product is perpendicular to both $\overline 1$ and to $(\overline{\alpha}_j)_{\perp}$. By the same reason, $\inner{\beta_i}{\overline 1}=0$ for every $i$.
Since for each $i$, $\beta_i(g)$ is a continuous function in $g$ and by Equation (\ref{2:06}) above, there exists some fixed basis $w_1,\dots,w_{n-1}$ of the lattice $(P_{V_0}\Lambda_{v_0}(E))^*$ such that:
\begin{equation}\label{late}
	\beta_i(g)=g_0^{t}g^{-1}g_0^{-t}w_i\text{ for all }i=1,\dots,n-1.
\end{equation}
Define $\Sigma_{n-1}=g_0^t\Sigma_{n-1}'g_0^{-t}$ and write $\Lambda_w=\operatorname{span}_{\ZZ}\{w_1,\dots,w_{n-1}\}$. We continue to calculate $\Gamma_E(\Lambda(v))$. Recall that by Definition \ref{rdef}, $\rho_n:X_n\rightarrow Y_{n-1}$ is an identification between lattices and grids and that by Definition \ref{notations}: $\Gamma_E(\Lambda(v))=\rho_n(\Lambda_v(E)^*)$. Write by Equation (\ref{15:04}) and by the definitions of $\beta_i,\tau_v$ (recall Equations (\ref{brei}),(\ref{tling})):
\begin{equation}
	(\Lambda_v(E))^*=\left(\begin{pmatrix}
	\mid\ & \cdots &\mid & \mid \\
	\alpha_1^0 & \cdots & \alpha_{n-1}^0& \lambda_vD_v\overline 1\\
	\mid & \cdots & \mid & \mid
	\end{pmatrix}\ZZ^n\right)^*=\begin{pmatrix}
	\mid\ & \cdots &\mid & \mid \\
	\alpha_1^0 & \cdots & \alpha_{n-1}^0& \lambda_vD_v\overline 1\\
	\mid & \cdots & \mid & \mid
	\end{pmatrix}^{-t}\ZZ^n=\begin{pmatrix}
	\mid\ & \cdots &\mid & \mid \\
	\beta_1 & \cdots & \beta_{n-1}& \tau_v\\
	\mid & \cdots & \mid & \mid
	\end{pmatrix}\ZZ^n
\end{equation}
which means that $(\Lambda_v(E))^*=\operatorname{span}_{\ZZ}\{\beta_1,\dots,\beta_{n-1},\tau_v\}$ and since $\beta_i\in V_0$ for every $i=1,\dots,n-1$, we can take $u=\tau_v$ in Definition \ref{rdef} of $\rho_n$ and deduce that the first layer is given by:
\begin{equation}\label{15:45}
	\Gamma_E(\Lambda(v))=\rho_n(\Lambda_v(E)^*)=\operatorname{span}_{\ZZ}\{\beta_1,\dots,\beta_{n-1}\}+P_{V_0}(\tau_v).
\end{equation}
We already know (see Corollary \ref{1:1}) that for any $v_0\in V_{\RR}^{(0)}$ (recall Remark \ref{orbs}) irreducible maximal points $v=t(g,g')v_0\in V_{\ZZ}\cap \RR_+\cdot (\Sigma'_{n-1},\Sigma_{r-1})v_0$ correspond under $v\mapsto R_{\ZZ}(v)$ to rings of integers of equivalence classes of type $S_n$ number fields $\FF$. Denote the number field corresponding to $v$ by $\FF_{v}$.
	 Recall that in Notation \ref{ord} we fixed (in particular) an ordering for the natural maps of $\FF_v$, whom we denote by $\sigma^v_1,\dots,\sigma^v_n$. Let $\pi(\FF_v)\in S_n$ be a permutation such that:
	\begin{equation}
		(\sigma^v_{\pi(\FF_v)(1)},\dots,\sigma^v_{\pi(\FF_v)(n)})\alpha_i(t(g,g')v_0)=\overline \alpha_i(t(g,g')v_0);\forall i=1,\dots,n
	\end{equation}
(recall \ref{basCons} for definition of $\overline{\alpha}_i$) and define:
\begin{equation}
	\Pi(\FF_v)=\pi(\FF_v).
\end{equation}
Let $v=t(g,g')v_0\in V_{\ZZ}\cap \RR_+\cdot (\Sigma_{n-1}',\Sigma_{r-1})v_0$ be irreducible maximal. Applying Equations (\ref{15:45}) and (\ref{2:06}) in particular to $v$ shows that:
\begin{equation}
	\Gamma^{\pi}_E(\FF_v)=\Gamma_E(\Lambda(v))=\operatorname{span}_{\ZZ}\{\beta_1,\dots,\beta_{n-1}\}+P_{V_0}(\tau_v)=g_0^{t}g^{-1}g_0^{-t}(\Lambda_{v_0}(E)^*\cap V_0)+P_{V_0}(\tau_v)
\end{equation}
which shows that $\Gamma^{\pi}_E(\FF_v)\in S\times_{\Sigma_{n-1}} U$ if and only if $g_0^{t}g^{-1}g_0^{-t}\in S$ and $P_{V_0}(\tau_v)\in g_0^{t}g^{-1}g_0^{-t}U\operatorname{mod}g_0^{t}g^{-1}g_0^{-t}\Lambda_w$ if and only if $g^{-1}\in g_0^{-t}Sg_0^{t}$ and:
\begin{equation}
	f_w(v)\vcentcolon=\frac{P_{V_0}(u)}{\inner{u}{1}}+\frac{1}{\inner{u}{1}}\sum_{i=1}^{n-1}\inner{u}{\overline{\alpha}_i(gv_0)}w_i\in g_0^{t}g^{-1}g_0^{-t}U\operatorname{mod}g_0^{t}g^{-1}g_0^{-t}\Lambda_w
\end{equation}
if and only if:
\begin{equation}
	v\in \{\text{irreducible, maximal points inside }[0,X] (g_0^tSg_0^{-t}\times \Sigma_{r-1})v_0\}\cap (f_w\operatorname{mod}\Lambda_w)^{-1}(U_{w})
\end{equation}
where the second equivalence holds by Equation (\ref{late}). Note that the only property of $\Pi$ we used is the continuity of $\beta_i(g)\in \RR^n$. Since we for any $\pi\in S_n$, the functions $\pi(\beta_i(g))$ are also continuous in $g$, the same argument shows that we could have used $\pi\circ\Pi$ as claimed.
\end{proof}
\newpage
\subsection{Weyl-Type Theorem and Verification of Analytical Properties}\label{ver}
For Theorem \ref{myWeyl} below, the following definition comes handy.
\begin{definition}\label{ED}
	Let $n,k\in \NN$, $F:D\subset \mathbb R^n\rightarrow \mathbb R^k$, $\Lambda_k\subset \RR^{k}$ a lattice, $\{B_R\}_{R\geq 0}$ a monotonic increasing family of subsets of $\RR^n$ and $\Gamma\subset \RR^n$ a discrete subset. We say that $F\operatorname{mod}\Lambda_k(\Gamma)$ is $B_R$-equidistributed if the sequence of discrete counting measures $\sigma^D_R$ of $\Gamma$ points in $B_R\cap D$ satisfies:
	\begin{equation}
		\lim_{R\rightarrow \infty}(F\operatorname{mod}\Lambda_k)_*\sigma^D_R=\lambda_k
	\end{equation} 
in the weak sense, where $F\operatorname{mod}\Lambda_k:D\rightarrow\RR^k/\Lambda_k$ is the composition of $F$ with the quotient map, $(F\operatorname{mod}\Lambda_k)_*\sigma^D_R$ denotes the push-forward measure, and $\lambda_k$ is the Haar probability measure on the $k$-dimensional torus $\RR^k/\Lambda_k$. When $\Lambda_k=\ZZ^k$ we abbreviate $F\operatorname{mod}\Lambda_k(\Gamma)=F\operatorname{mod}1(\Gamma)$.
\end{definition}
The following Theorem from ~\cite{myNote} is the main tool in the proof of Theorem \ref{torsion2}.

\begin{theorem}[~\cite{myNote}]\label{myWeyl}
	Let $n,D,k\in \NN$. Let $S$ be a level surface of some smooth homogeneous function $F:\RR^n\setminus\{\overline 0\}\rightarrow \RR_+$ (namely $S=\{x\in \RR^n:F(x)=t\}$ for some $t>0$), let $A\subset S$ be open and Jordan measurable in $S$ and let $f=(f_1,\dots,f_k):\RR^n\setminus \{\overline 0\}\rightarrow \RR^k$ be homogeneous of degree $D$ and smooth inside $\RR_+\cdot A$. Let $\Gamma\subset \RR^n$ be a lattice. Assume that there exists $v\in \Gamma$ such that:
	\begin{equation*}
		\left(\frac{\partial^d f}{\partial v^d}\operatorname{mod}1\right)_*m_A\vcentcolon=\left(\frac{\partial^d f_1}{\partial v^d}\operatorname{mod}1,\dots,\frac{\partial^d f_k}{\partial v^d}\operatorname{mod}1\right)_*m_A\ll \lambda_k
	\end{equation*}
where $m_A$ is the $n-1$-dimensional surface area on $A$ and $\lambda_k$ is the Lebesgue measure on $(S^1)^k$. Then:
\begin{equation}\label{15:23}
	f\operatorname{mod}1(\Gamma)\text{ is $[0,R]\cdot A$-equidistributed}.
\end{equation}
For a subset $I\subset \RR_+$, $I\cdot A$ denotes the set $\{ta:t\in I,a\in A\}$.
\end{theorem}
\begin{definition}\label{DeCo}
Let $S\subset \RR^d$ be a level surface of some smooth homogenous function. 
	\begin{enumerate}
		\item A subset $D_{\ZZ}\subset \ZZ^d$ is said to be of $S$-density 1 in $\ZZ^d$ if:
	\begin{equation}
		\lim_{R\rightarrow \infty}\frac{\abs{(([0,R]\cdot S)\cap D_{\ZZ})}}{\operatorname{Vol}_{d}([0,1]\cdot S)R^d}=1;
	\end{equation}
		\item A subset $C_{\ZZ}\subset \ZZ^d$ is said to be a {finite} congruence set if it is defined by union of finitely many congruence conditions modulo coprime numbers on the entries of $v\in C_{\ZZ}$ (namely conditions of the form $v_i\operatorname{mod}p=s$ where $v=(v_1\dots,v_n)$, $p,s$ are integers and $p$ is taken from a set of coprime numbers). Similarly, a set $C_{\ZZ}\subset \ZZ^d$ is said to be an {infinite} congruence set if it is defined by infinitely many congruence conditions as above. 
	\end{enumerate}
\end{definition}
We will require a modification of Theorem \ref{myWeyl}:
\begin{Corollary}\label{CRT}
	Under the assumptions of Theorem \ref{myWeyl} Equation (\ref{15:23}) holds for $C_{\ZZ}\cap D_{\ZZ}$ instead of $\Gamma$ where $D_{\ZZ}$ is any set of $S$-density $1$ in $\ZZ^d$ and $C_{\ZZ}$ is any finite congruence set.
\end{Corollary}
\begin{proof}
	By taking the average, it suffices to prove the case where $C_{\ZZ}$ is given by precisely one condition modulo each reminder. Let $p_1,\dots,p_k$ be the coprime numbers defining the congruence conditions on the first coordinate of elements in $C_{\ZZ}$ and let $s_1,\dots,s_k$ be the corresponding reminders. Let $P=p_1\cdot\dots\cdot p_k$ and for any $i$, $P_i=P/p_i$. Since the $p_i$'s are coprime, $P_i,p_i$ are coprime. By Bezout's identity there exist integers $M_i,m_i$ such that:
	\begin{equation}
		M_iP_i+m_ip_i=1.
	\end{equation}
	Then $x_1\vcentcolon=\sum_{i=1}^ks_iM_iP_i$ is a solution to the congruence system of conditions on the first entries of elements from $C_{\ZZ}$. Since $p_i$'s are distinct primes, by the Chinese remainder theorem any other solution to the system is given by $x_1+lP$ for some integer $l$. Repeating this argument for any other coordinate we can write:
	\begin{equation}\label{16:04}
		C_{\ZZ}=\{(x_1+l_1Q_1,\dots,x_d+l_dQ_d):(l_1,\dots,l_d)\in \ZZ^d\}
	\end{equation} 
	for some fixed $x_i,Q_i\in \ZZ$. Note that there exists $c>0$ such that for any $R>0$:
	\begin{equation}\label{16:08}
		\abs{C_{\ZZ}\cap [0,R]\cdot S}\geq cR^d.
	\end{equation}
Let $F$ satisfy the conditions of Theorem \ref{myWeyl}. Then the function $G(y_1,\dots,y_d)\vcentcolon=F(x_1+y_1Q_1,\dots,x_d+y_dQ_d)$ also satisfies the conditions of Theorem \ref{myWeyl} (since $x_i,Q_i$ are integers) and therefore by Equation (\ref{16:04}):
	\begin{equation}
			F\operatorname{mod}1(C_{\ZZ})\text{ is }[0,R]\cdot A\text{ equidistributed}.
	\end{equation}
where $A$ is as in Theorem \ref{myWeyl}. By assumption on $D_{\ZZ}$:
\begin{equation}
	\lim_{R\rightarrow\infty}\frac{\abs{(([0,R]\cdot S)\cap \ZZ^d)\setminus D_{\ZZ}}}{R^d}=0
\end{equation}
so by Equation (\ref{16:08}):
\begin{equation}
	\lim_{R\rightarrow\infty}\frac{(([0,R]\cdot S)\cap C_{\ZZ})\setminus D_{\ZZ}}{\abs{C_{\ZZ}\cap [0,R]S}}=0
\end{equation}
which means that $D_{\ZZ}^c$ points are negligible in $C_{\ZZ}$ so in fact:
\begin{equation}
		F\operatorname{mod}1(C_{\ZZ}\cap D_{{\ZZ}})\text{ is }[0,R]\cdot A\text{ equidistributed}
\end{equation}
as desired. 
\end{proof}
\begin{Lemma}\label{BB}
	The subset $\Gamma_d\subset V_{\ZZ}$ of irreducible maximal lattice points in $V_{\ZZ}$ (recall Definition \ref{irma}) is the intersection of a set of $S=\Sigma'_{n-1}\times\Sigma_{r-1}v_0$ density $1$ in $\ZZ^d$ and an infinite congruence set. 
\end{Lemma}
\begin{proof}
By ~\cite[Theorem 4]{BhaPip} and the short discussion following it we know that the irreducible points are of $S$ density 1. By ~\cite[Section 5]{BhaPip}, the maximal points are defined by infinitely many congruence conditions modulo prime powers. 	
\end{proof}
\begin{Corollary}\label{kefel}
	Under the assumptions of Theorem \ref{myWeyl} Equation (\ref{15:23}) holds for $\Gamma_d$ instead of $\ZZ^d$. 
\end{Corollary}
\begin{proof}
	Write $\Gamma_d=C_d\cap D_d$ where $D_d$ is a set of $S$ density 1 and $C_d$ is some infinite congruence set. For any prime $p$ let $U_p$ be the subset of $C_d$ where the congruence conditions involve only $p$. By ~\cite[Lemma 11]{BhaPip}:
	\begin{equation}
		\abs{\left([0,R]S\cap \ZZ^d\right)\setminus U_p}=O(R^d/p^2)
	\end{equation}
and by ~\cite[Section 5]{BhaPip} there exists $\alpha>0$ such that:
\begin{equation}
	\abs{[0,R]S\cap \Gamma_d}=\alpha R^d\operatorname{Vol}([0,1]S)+o(R^d).
\end{equation}
Let $\eps>0,R>0$ and let $Y_0$ to be later determined. Then:
\begin{equation}\label{skorka}
	\frac{\abs{[0,R]S\cap \left(\Gamma_d\setminus \cup_{p<Y_0}U_p\right)}}{\abs{[0,R]S\cap\left(\Gamma_d\setminus \cup_{p>Y_0}U_p\right)}}\leq \frac{\sum_{p\geq Y_0}O(R^d/p^2)}{\abs{[0,R]S\cap \Gamma_d}}=\frac{\sum_{p\geq Y_0}O(R^d/p^2)}{\alpha R^d\operatorname{Vol}([0,1]S)+o(R^d)}\leq \eps
\end{equation}
for any $Y_0$ large enough. Let $k,f$ be as in Theorem \ref{myWeyl} and let $W\subset \RR^k/\ZZ^k$ be an open subset. Denote for any $R>0$ and $Y>0$:
\begin{equation}
	N(\Gamma_d^{<Y},A,R)=\abs{[0,R]S\cap (f\operatorname{mod}1)^{-1}(W)\cap \left(\cup_{p<Y}U_p\cap \Gamma_d\right)}
\end{equation}
and similarly define $N(\Gamma_d^{>Y},A,R)$ and $N(\Gamma_d^{<Y},R)=N(\Gamma_d^{<Y},\RR^k/\ZZ^k,R)$,$N(\Gamma_d^{>Y},R)=N(\Gamma_d^{>Y},\RR^k/\ZZ^k,R)$. Then to prove the Corollary we need to estimate for any open Jordan measurable $W\subset \RR^k/\ZZ^k$:
\begin{equation}
	\frac{N(\Gamma_d^{<Y},W,R)+N(\Gamma_d^{>Y},W,R)}{N(\Gamma_d^{<Y},R)+N(\Gamma_d^{>Y},R)}=\frac{\frac{N(\Gamma_d^{<Y},W,R)}{N(\Gamma_d^{<Y},R)}+\frac{N(\Gamma_d^{>Y},W,R)}{N(\Gamma_d^{<Y},R)}}{1+\frac{N(\Gamma_d^{>Y},R)}{N(\Gamma_d^{<Y},R)}}=\frac{N(\Gamma_d^{<Y},W,R)}{N(\Gamma_d^{<Y},R)}+\eps
\end{equation}
for all $R,Y$ large enough, by Equation (\ref{skorka}). Since $\Gamma_d^{<Y}$ is defined by finitely many congruence conditions, Corollary \ref{CRT} implies that $\frac{N(\Gamma_d^{<Y},W,R)}{N(\Gamma_d^{<Y},R)}\rightarrow \operatorname{Vol}_{k}(W)$ as $R\rightarrow \infty$. Taking $\eps\rightarrow 0$ we deduce:
\begin{equation}
	\frac{N(\Gamma_d^{<Y},W,R)+N(\Gamma_d^{>Y},W,R)}{N(\Gamma_d^{<Y},R)+N(\Gamma_d^{>Y},R)}\rightarrow \operatorname{Vol}_{k}(W)
\end{equation}
and the claim of Theorem \ref{myWeyl} is proved. 
\end{proof}
Our main result of this Section will be stated using the function $f_n^E$ guaranteed by Lemma \ref{structure}. For convenience and clarity, we define it again:
\begin{definition}\label{funcDef}
	For any hyperplane $E=u^{\perp}\subset \RR^n$, $v_0\in V_{\RR}^{(0)}$ and $\Sigma_{n-1}'$ be a fundamental domain for the action of $\operatorname{SL}_{n-1}(\ZZ)$ on $\operatorname{SL}_{n-1}(\RR)$ which we will use throughout this Section, let $f^E_n$ denote the function corresponding to the subspace $E$ under Lemma \ref{structure} and let $(w_i)_{i=1}^{n-1}$ be the fixed basis of $V_0$ also defined in Lemma \ref{structure}. Explicitly, we denote:
	\begin{equation*}
		f_n^E(p)=\frac{1}{\inner{u}{\overline{1}}}\sum_{i=1}^{n-1}\inner{u}{\overline \alpha_i(p)}w_i+q_0
	\end{equation*}
for any $p\in \RR_+\cdot\Sigma_{n-1}\times \Sigma_{r-1} v_0$. Moreover, denote:
\begin{equation}
	\Lambda_w=\operatorname{span}_{\ZZ}\{w_1,\dots,w_{n-1}\}.
\end{equation}
\end{definition}
The following Lemma summarizes the properties of the function $f_n^E$ which are relevant for Section \ref{PF}.
\begin{Lemma}\label{calc}
	Let $n=3$ ($n=4$,$n=5$), let $u\in \RR^n$ satisfying $u\notin (\RR\cdot \overline 1)\cup \overline 1^{\perp}$ and let $v_0\in V_{\RR}^{(0)}$. Then $f_n^{E=u^{\perp}}$ is homogeneous of degree $D=1$ ($D=2$,$D=4$) and for any Jordan measurable bounded open subset $U$ of $S=\Sigma_{n-1}'\times \Sigma_{r-1}v_0$ it holds that $f_n^E,U,S$ satisfy the conclusion of Theorem \ref{myWeyl} for $k=n-1$. Namely:
\begin{equation}\label{claim}
	f_n^E\operatorname{mod}\Lambda_w(\Gamma_d)\text{ is $[0,R]\cdot U$-equidistributed}
\end{equation}
where $\Gamma_d\subset V_{\ZZ}$ is the set of irreducible maximal lattice points in $V_{\ZZ}$ (recall Definition \ref{irma}).
\end{Lemma}
Since there is an overlap in the arguments of the proof for $n=3,4,5$, we write this overlap once and then continue to prove each case separately. Let $u\in \RR^n$ such that $u\notin(\RR\cdot \overline 1)\cup \overline 1^{\perp}$ and fix $v_0\in V^{(0)}_{\RR}$. Let $U\subset S=\Sigma_{n-1}'\times \Sigma_{r-1}v_0$ be open and Jordan measurable subset.  
	It holds (e.g. by Weyl's criterion for uniform distribution in ~\cite{Weyl}) that (\ref{claim}) is equivalent to:
	\begin{equation}\label{claim2}
		\text{For any }v_{\ZZ}\in \Lambda_w\text{ }\inner{f_n^E}{v_{\ZZ}}\operatorname{mod}1(\Gamma_d)\text{ is $[0,R]\cdot U$-equidistributed}.
	\end{equation}
Suppose for sake of contradiction that $v_{\ZZ}\in \Lambda_w$ is such that (\ref{claim2}) does not hold. Denote $f_0=\inner{f_n^E}{v_{\ZZ}}$. Note that by Definition \ref{funcDef} there exist real constants $(\gamma_1,\dots,\gamma_{n-1})\neq 0$ such that:
\begin{equation}
	f_0(p)=\sum_{i=1}^{n-1}\gamma_i\inner{\overline\alpha_i(p)}{u}
\end{equation}
for every $p\in \RR_+\cdot U\subset \RR^d$. By Lemma \ref{BB}, $\Gamma_d$ is the intersection of an infinite congruence set and a set of $S$ density $1$ in $\ZZ^d$. Therefore, by Corollary \ref{kefel} we may assume that for any $v\in \Gamma_d$, the function $f_0$, $S,A$ and $v$ do not satisfy the conditions of Theorem \ref{myWeyl}. Indeed, otherwise the Corollary will ensure that (\ref{claim2}) holds. Explicitly, it means that for any $v\in \Gamma_d$ the measure $(\frac{\partial^{D(n)} f_0}{\partial v^{D(n)}}\operatorname{mod}1)_*m_A$ is not absolutely continuous w.r.t $\lambda$ (recall Definition \ref{ED}). Lemma \ref{lesles} implies that there exists a subset $B=B_v\subset U$ of positive $m_U$ measure such that $\nabla\frac{\partial^{D(n)} f_0}{\partial v^{D(n)}}=0$ on $B$. By Subsection \ref{basCons} the smooth functions $\alpha_1,\dots,\alpha_{n-1}:\RR_+\cdot U\rightarrow \RR^n$ satisfy the multiplication table \ref{MT3} when $n=3$, \ref{MT4} when $n=4$ and \ref{MT5} when $n=5$. Therefore, by Lemma \ref{reA} for the open set $V=\RR_+\cdot U$ and $\alpha_1,\dots,\alpha_{n-1}$ we deduce that $\alpha_1,\dots,\alpha_{n-1}$ are real analytic in $V$. Since linear combinations and derivatives of real analytic functions is again real analytic, this shows that $f_0$ and $\nabla \frac{\partial^{D(n)} f_0}{\partial v^{D(n)}}$ are real analytic. By Fubini's Theorem $[1/2,3/2]\cdot B$ has positive $d$-dimensional Lebesgue measure so Lemma \ref{posM} for $\nabla \frac{\partial^{D(n)} f_0}{\partial v^{D(n)}}$ and $[1/2,3/2]\cdot B\subset V$ implies that: 
\begin{equation}\label{17:06}
	 \nabla \frac{\partial^{D(n)} f_0}{\partial v^{D(n)}}=0\text{ on }V.
\end{equation}
Indeed, since $\nabla \frac{\partial^{D(n)} f_0}{\partial v^{D(n)}}$ is homogeneous of degree $-1$, $\nabla \frac{\partial^{D(n)} f_0}{\partial v^{D(n)}}\mid_{B}=0$ implies $\nabla \frac{\partial^{D(n)} f_0}{\partial v^{D(n)}}\mid_{[1/2,3/2]\cdot B}=0$.
\begin{proof}[Proof of Lemma \ref{calc}, case $n=3$]
We continue from Equation (\ref{17:06}). Since $v\in \Gamma_d$ was arbitrary and $\Gamma_d$ contains a basis for $\RR^d$, we can write:
\begin{equation}
	f_0(a_1,\dots,a_4)=\sum_{i=1}^4\eta_ia_i
\end{equation}
for some constants $\eta_i\in \RR;i=1,\dots,4$ and for every $(a_1,\dots,a_4)\in \RR_+\cdot U$. In other words, for every $g=\begin{pmatrix}
	\ r& s \\
	t & u \\
	\end{pmatrix}\in U$:
	\begin{equation*}
		\gamma_1\inner{\overline\alpha_1(gv_0)}{u}+\gamma_2\inner{\overline\alpha_2(gv_0)}{u}=\sum_{i=1}^4\eta_ia_i(gv_0)
	\end{equation*}
	where for $v\in V_{\RR}$ we denote $v=(a_1(v),\dots,a_4(v))$. Since $\inner{\overline \alpha_1(gv_0)}{\overline 1},\inner{\overline \alpha_2(gv_0)}{\overline 1}$ are proportional to $a_2(gv_0),a_3(gv_0)$ respectively (see ~\cite[p.35]{PipThes}), we get (possibly for different constants $\gamma_i,\eta_i$):
	\begin{equation}\label{AA}
		\gamma_1\inner{P_{V_0}\overline\alpha_1(gv_0)}{u}+\gamma_2\inner{P_{V_0}\overline\alpha_2(gv_0)}{u}=\sum_{i=1}^4\eta_ia_i(gv_0).
	\end{equation}
	Write $\kappa_i=\inner{P_{V_0}\overline \alpha_i(v_0)}{u};i=1,2$, then as follows from Subsection \ref{basCons}:
	\begin{equation}\label{B1}
		\inner{P_{V_0}\overline \alpha_1(gv_0)}{u}=(ru-st)(\kappa_1r+\kappa_2s),
	\end{equation}
	\begin{equation}\label{B2}
		\inner{P_{V_0}\overline \alpha_2(gv_0)}{u}=(ru-st)(\kappa_1t+\kappa_2u).
	\end{equation}
	By the first equation at the bottom of ~\cite[p.38]{PipThes} we get:
	\begin{equation*}
		\sum_{i=1}^4\eta_ia_i(gv_0)=\eta_1(a_0r^3+b_0r^2s+c_0rs^2 +d_0s^3)+\eta_2(3a_0r^2t+2b_0rst+c_0s^2t+b_0r^2u+2c_0rsu+3d_0s^2u)
	\end{equation*}
	\begin{equation*}
		+\eta_3(3a_0rt^2 + 2b_0rtu + b_0st^2 + c_0ru^2 + 2c_0stu + 3d_0su^2)+\eta_4(a_0t^3 + b_0t^2u + c_0tu^2 + d_0u^3)
	\end{equation*}
	Assume $a_0\neq 0$. Then since neither of $r^3,r^2t,rt^2,t^3$ appear in Equations (\ref{B1}),(\ref{B2}), Equation (\ref{AA}) implies  that $\eta_i=0;i=1,\dots,4$, a contradiction. Similarly, assuming that $d_0\neq 0$, since $s^3,s^2u,su^2,u^3$ do not appear in Equations (\ref{B1}),(\ref{B2}) we get the same contradiction. Since $\operatorname{Disc}(v_0)=b_0^2c_0^2-4a_0c_0^3-4b_0^3d_0-27a_0^2d_0^2 + 18a_0b_0c_0d_0\neq 0$ by assumption, we find that $b_0,d_0\neq 0,a_0=d_0=0$. The coefficients of $rst,r^2u$ in the right hand side of (\ref{AA}) are $2\eta_2b_0,\eta_2b_0$ respectively. On the left hand side they are $-\gamma_1\kappa_1,\gamma_1\kappa_1$ respectively. This implies that $-\gamma_1\kappa_1=2\gamma_1\kappa_1$ thus $\kappa_1\gamma_1=0$. Similarly, we show that $\kappa_2\gamma_2=0$. Note that by our choice of $u$, $\kappa_i=\inner{P_{V_0}\overline \alpha_i(v_0)}{u}\neq 0;i=1,2$ so that $\gamma_1=\gamma_2=0$, a contradiction. 
\end{proof}
\begin{proof}[Proof of Lemma \ref{calc}, case $n=4$]
We continue from Equation (\ref{17:06}). Since $v\in \Gamma_d$ was arbitrary and $\Gamma_d$ contains a basis for $\RR^d$, we can write:
\begin{equation}\label{17:24}
	f_0(a_1,\dots,a_4)=\sum_{i,j=1}^{6}\eta_{i,j}a_i(gv_0)a_j(gv_0)
\end{equation} 
where $a_i(v);i=1,\dots,12$ are the coefficients of an element $v\in V_{\RR}$. Fix $v_0=(A_0,B_0)\in V^{(0)}_{\RR}$ where $A_0,B_0$ is a pair of symmetric $3\times3$ matrices (see ~\cite[p.46]{PipThes}). Write:

\begin{equation}\label{symM}
h\vcentcolon=A_0=\begin{pmatrix}
	a_0 & d_0/2 &f_0/2 \\
	d_0/2 & b_0 & e_0/2\\
	f_0/2 &e_0/2 & c_0
	\end{pmatrix}.
\end{equation}
For a general $v\in V_{\RR}^{(0)}$ write $v=(A(v),B(v))$ where $A(v),B(v)$ is a pair of symmetric $3\times 3$ matrices (see ~\cite[p.46]{PipThes}). Equation (\ref{17:24}) is equivalent to:
	\begin{equation*}
		\gamma_1\inner{\overline\alpha_1(gv_0)}{u}+\gamma_2\inner{\overline\alpha_2(gv_0)}{u}+\gamma_3\inner{\overline\alpha_3(gv_0)}{u}=\sum_{i,j=1}^{6}\eta_{i,j}a_i(gv_0)a_j(gv_0).
	\end{equation*}
Since $\inner{\overline \alpha_i(gv_0)}{\overline 1},i=1,2,3$ can each be written as some linear combination of $a_i(gv_0)a_j(gv_0);i,j=1,\dots,6$ (see ~\cite[p. 41-42]{PipThes}), we get (possibly for different constants):
\begin{equation}\label{C0}
	\gamma_1\inner{P_{V_0}\overline\alpha_1(gv_0)}{u}+\gamma_2\inner{P_{V_0}\overline\alpha_2(gv_0)}{u}+\gamma_3\inner{P_{V_0}\overline\alpha_3(gv_0)}{u}=\sum_{i,j=1}^{6}\eta_{i,j}a_i(gv_0)a_j(gv_0).
\end{equation}
Write $\kappa_i=\inner{P_{V_0}\overline \alpha_i(v_0)}{u};i=1,2,3$ and $g=\begin{pmatrix}
	-& v_1 &- \\
	- & v_2 & -\\
	- &v_3 &-
	\end{pmatrix}$ then it follows from Subsection \ref{basCons}:
	\begin{equation}\label{C1}
		\inner{P_{V_0}\overline \alpha_1(gv_0)}{u}=\det(g)\inner{v_1}{\overline \kappa},
	\end{equation}
	\begin{equation}\label{C2}
		\inner{P_{V_0}\overline \alpha_2(gv_0)}{u}=\det(g)\inner{v_2}{\overline \kappa},
	\end{equation}
	\begin{equation}\label{C3}
		\inner{P_{V_0}\overline \alpha_3(gv_0)}{u}=\det(g)\inner{v_3}{\overline \kappa}.
	\end{equation}
	where $\overline \kappa=(\kappa_1,\kappa_2,\kappa_3)$. If $g\in U$, by ~\cite[p.46]{PipThes}:
	\begin{equation}
		A(gv_0)=gA_0g^t=\begin{pmatrix}
	-& v_1 &- \\
	- & v_2 & -\\
	- &v_3 &-
	\end{pmatrix}
	\begin{pmatrix}
	a_0 & d_0/2 &f_0/2 \\
	d_0/2 & b_0 & e_0/2\\
	f_0/2 &e_0/2 & c_0
	\end{pmatrix}
	\begin{pmatrix}
	\textemdash & \mid &\mid \\
	v_1 & v_2 & v_3\\
	\mid &\mid &\mid
	\end{pmatrix}=\begin{pmatrix}
	-& v_1 &- \\
	- & v_2 & -\\
	- &v_3 &-
	\end{pmatrix}\begin{pmatrix}
	\inner{h_1}{v_1} & \inner{h_1}{v_2} &\inner{h_1}{v_3} \\
	\inner{h_2}{v_1} & \inner{h_2}{v_2} & \inner{h_2}{v_3}\\
	\inner{h_3}{v_1} &\inner{h_3}{v_2} & \inner{h_3}{v_3}
	\end{pmatrix}.
	\end{equation}
A general entry of the matrix above will be denoted by $(i,j)\vcentcolon=\inner{v_i}{(\inner{h_1}{v_j},\inner{h_2}{v_j},\inner{h_3}{v_j})}$.	Therefore, by Equations (\ref{C0}),(\ref{C1}),(\ref{C2}) and (\ref{C3}) we get:
\begin{equation}\label{H}
	\det(g)(\gamma_1\inner{v_1}{\overline \kappa}+\gamma_2\inner{v_2}{\overline \kappa}+\gamma_3\inner{v_3}{\overline \kappa})=\sum_{i_1,j_1,i_2,j_2=1,\dots,6}\eta_{i_1,i_2,j_1,j_2}(i_1,j_1)(i_2,j_2). 
\end{equation}
Write $v_i=(v_{i1},v_{i2},v_{i3})$ for every $i=1,2,3$. Assume without loss of generality that $\kappa_1\neq0$ and $\gamma_1\neq0$ (it will be clear from the argument why this assumption preserves the generality). Consider the coefficients $c_1,c_2$ of $v_{11}^2v_{22}v_{33}$ and $v_{11}^2v_{23}v_{32}$ respectively on the left hand side of Equation (\ref{H}). By definition of $\det(g)$ and Equation (\ref{H}) it holds that $c_1=-c_2=\kappa_1\gamma_1\neq 0$.  However, direct calculation shows that the coefficient of $v_{11}^2v_{22}v_{33}$ on the right hand side of Equation $(\ref{H})$ is $\eta_{1,1,2,3}h_{11}h_{23}+\eta_{1,3,1,2}h_{13}h_{12}$. The same calculation shows that $\eta_{1,1,2,3}h_{11}h_{32}+\eta_{1,3,1,2}h_{12}h_{13}$ is the coefficient of $v_{11}^2v_{23}v_{32}$. However, $h$ is symmetric (recall Equation (\ref{symM})) so these two expressions are equal. Under the assumption that Equation (\ref{H}) poses we got that on the one hand $c_1=-c_2$ and on the other $c_1=c_2$. The inevitable conclusion that $c_1=c_2=0$ contradicts our assumption that $\kappa_1,\gamma_1\neq 0$. Indeed, the choice $\kappa_1,\gamma_1\neq 0$ did not affect the generality: we would have otherwise onsidered the coefficients of $v_{i,j}^2v_{kl}v_{mp},v_{i,j}^2v_{kp}v_{ml}$ for some other appropriate $i,j,k,l,m,p\in \{1,2,3\}$.
\end{proof}	
\begin{proof}[Proof of Lemma \ref{calc}, case $n=5$]
We continue from Equation (\ref{17:06}). Denote:
\begin{equation}
	I=\{(i,j,k):i,j=1,\dots,5,k=1,\dots,4\}
\end{equation}
and for $i_0=(i,j,k)\in I$ denote:
\begin{equation}
A(i_0)=A^{(k)}_{i,j}.	
\end{equation}
Since $v\in \ZZ^d$ was arbitrary, we can write:
\begin{equation}\label{2:45}
	f_0(A^{(1)},\dots,A^{(4)})=\sum_{i_1,\dots,i_5\in I}\eta_{i_1,\dots,i_5}A(i_1)A(i_2)A(i_3)A(i_4)A(i_5)
\end{equation} 
Write $g_4^t=\begin{pmatrix}
	\ \mid& \mid &\mid &\mid \\
	v_1 & v_2 & v_3 & v_4\\
	\mid &\mid &\mid &\mid\end{pmatrix},g_5^t=\begin{pmatrix}
	\ \mid& \mid &\mid &\mid &\mid \\
	u_1 & u_2 & u_3 & u_4& u_5\\
	\mid &\mid &\mid &\mid&\mid\end{pmatrix}$ and $g=(g_4,g_5)$. By ~\cite[p. 52]{PipThes} we know that:
\begin{equation}\label{2:46}
	g\cdot v_0=\left(g_4\colf{g_5A^{(1)}g_5^{t}}{g_5A^{(2)}g_5^{t}}{g_5A^{(3)}g_5^{t}}{g_5A^{(4)}g_5^{t}}\right)
\end{equation}
where we abuse the notation $g=(g_4,g_5)$. Now, given $u$ as in the conditions of the Lemma we write by \ref{basCons}:
\begin{equation}\label{22:10}
	\inner{P_{V_0}(\alpha_i(gv_0)}{u}=\det(g_4)\left(\sum_{j=1}^4v_{ij}\inner{P_{V_0}(\alpha_j(v_0))}{u}\right),
\end{equation}
\begin{equation}\label{22:11}
	\inner{P_{\RR\cdot \overline 1}\alpha_i(gv_0)}{u}=\frac{1}{n}\inner{\alpha_i(gv_0)}{\overline 1}\inner{u}{\overline 1}.
\end{equation}
Equation (\ref{2:45}) and the above two equations imply:  
\begin{equation*}
	\sum_{i_1,\dots,i_5\in I}\eta_{i_1,\dots,i_5}\left(A(i_1)A(i_2)A(i_3)A(i_4)A(i_5)\right)(gv_0)=f_0(A^{(1)}(gv_0),A^{(2)}(gv_0),A^{(3)}(gv_0),A^{(4)}(gv_0))
\end{equation*}
\begin{equation}\label{22:38}
	=\inner{\overline{\alpha}_i(gv_0)}{1}+\frac{\det{(g_4)}}{\inner{u}{1}}\sum_{j=1}^4v_{ij}\inner{P_{V_0}(\alpha_j(v_0))}{u}+q_0
\end{equation}
and so:
\begin{equation}
	\sum_{i_1,\dots,i_5\in I}\eta_{i_1,\dots,i_5}\left(A(i_1)A(i_2)A(i_3)A(i_4)A(i_5)\right)(gv_0)-\inner{\overline{\alpha}_i(gv_0)}{1}=\frac{\det{(g_4)}}{\inner{u}{1}}\sum_{j=1}^4v_{ij}\inner{P_{V_0}(\alpha_j(v_0))}{u}+q_0
\end{equation}
which, by Equation (\ref{2:46}) is a contradiction, the left hand side of the above equation has non-trivial dependence in $g_5$ while the right hand side doesn't.
\end{proof}
\section{Proofs}\label{PF}
In this Section we prove Theorems \ref{torsion2} and \ref{unif3} (in this order) by proving their equivalent formulations, namely Theorems \ref{unif3g} and \ref{torsions2g} repectively. Let $n=3,4,5$ and fix fundamental domains $\Sigma_{n-1}',\Sigma_{r-1}$ for the actions of $\operatorname{SL}_{n-1}(\ZZ)$ and $\operatorname{SL}_{r-1}(\ZZ)$ on $\operatorname{SL}_{n-1}(\RR)$ and $\operatorname{SL}_{n-1}(\RR)$ respectively. Since we use it in the proof, we remind the reader of Definition \ref{funcDef} by repeating it once again:
\begin{definition*}
For any hyperplane $E=u^{\perp}\subset \RR^n$, let $f^E_n$ denote the function corresponding to the subspace $E$ under Lemma \ref{structure}. Explicitly, we denote:
	\begin{equation*}
		f_n^E(p)=\frac{1}{\inner{u}{\overline{1}}}\sum_{i=1}^{n-1}\inner{u}{\overline \alpha_i(p)}w_i+q_0
	\end{equation*}
for any $p\in \RR_+\cdot(\Sigma_{n-1}\times \Sigma_{r-1}) v_0$ (recall that in Definition \ref{funcDef} $(w_i)_{i=1}^{n-1}$ was some fixed basis of $V_0$).
\end{definition*}
Before proving the main Theorems we make a reduction that will serve us both in the proof of Theorem \ref{unif3g}: First, for convenience of this discussion, given a choice of orderings namely $\Pi:\mathcal F^{(0)}_{\infty}\rightarrow S_n$ (recall Subsection \ref{Not} for the definition of $\mathcal F^{(0)}_{\infty}$), write $\operatorname{ED}(\Pi)$ if (recall Subsection \ref{NLN} for the definition of $\Gamma^{\Pi(\FF)}_E$):
\begin{equation*}
	\Gamma^{\Pi(\FF)}_E\text{ is equidistributed in }Y_{n-1}\text{ when }\FF\in  \mathcal F_n \text{ are ordered by discriminant}.
\end{equation*}
For $n=3,4,5$, if there exists $\Pi:\mathcal F_{\infty}^{(0)}\rightarrow S_n$ such that $\operatorname{ED}(\Pi)$ and for any $\pi\in S_n$, $\operatorname{ED}(\pi\circ\Pi)$ where $(\pi\circ\Pi)(\FF)\vcentcolon=\pi\circ (\Pi(\FF))$ then averaging the counting measures for each $\pi$ proves the equidistribution statement in Theorem \ref{unif3g}. Since we find such $\Pi$ using Lemma \ref{structure}, the extra property $\operatorname{ED}(\pi\circ\Pi)$ for any $\pi\in S_n$ will hold by the last part of this Lemma with the same argument. We continue with this reduction in mind.
\begin{proof}[Proof of Theorem \ref{unif3g}]\label{PF}
Let $\overline 1\neq u\in \RR^n$ such that $u\notin V_0$ and denote $E=u^{\perp}$. Fix $i=1,\dots,\floor{n/2}$, pick $v_0\in V_{\RR}^{(i)}$ and let $\Pi:\mathcal F_X^{(0)}\rightarrow S_n$ and $\Sigma_{n-1}$ be the function and fundamental domain (for the action of $\operatorname{stab}_{\operatorname{SL}_n(\RR)}(\Lambda_w)$ on $\operatorname{SL}_{n-1}(\RR)$) which correspond to $E$, $v_0$ and $i$ by Lemma \ref{structure}. It suffices to prove the claim for any signature separately because then the claim follows by taking the average. Since our discussion (particularly Theorem \ref{structure}) holds for any $i$, we assume without the loss of generality that $i=0$ and we are in the totally real case. Abbreviate $f_E=f_n^E$. For an open bounded and Jordan measurable subset $S\subset \Sigma_{n-1}$ and $V=I_1w_1+\dots+I_{n-1}w_{n-1}$ where $I_i\subset [0,1]$ are open intervals, let $A=S\times_{\Sigma_{n-1}}V\subset Y_{n-1}$ be a $\Lambda_w$-$\Sigma_{n-1}$-basic subset (recall Definition \ref{BS}). From here on, we abbreviate "open bounded and Jordan measurable" to OBJOM.\\
For every $\eps>0$ let $\Sigma_{r-1}^{\eps}\subset \Sigma_{r-1}$ be OBJOM such that:
\begin{equation}\label{18:03}
	\operatorname{Vol}\left([0,1]\cdot (S\times\Sigma_{r-1})v_0\right)\geq \operatorname{Vol}\left([0,1]\cdot (S\times\Sigma_{r-1}^{\eps})v_0\right)\geq (1-\eps)\operatorname{Vol}\left([0,1]\cdot (S\times\Sigma_{r-1})v_0\right).
\end{equation}
For every $T>0$ denote:
\begin{equation}
	S_T^{\eps}=[0,T]\cdot \left(S\times \Sigma_{r-1}^{\eps}v\right)v_0,S_T=[0,T]\cdot \left(S\times \Sigma_{r-1}\right)v_0;
\end{equation}
\begin{equation}
	S_{T,\ZZ}^{\eps}=S_T^{\eps}\cap \{\text{irreducible maximal points}\},S_{T,\ZZ}=S_T\cap \{\text{irreducible maximal points}\};
\end{equation}
\begin{equation}
	L_T=[0,T]\cdot \left(\Sigma_{n-1}\times \Sigma_{r-1}\right)v_0,L_{T,\ZZ}=L_T\cap \{\text{irreducible maximal points}\};
\end{equation}
\begin{equation}
	p_T^{\eps}(V)=\frac{\abs{S_{T,\ZZ}^{\eps}\cap f_E^{-1}(V)}}{\abs{S_{T,\ZZ}^{\eps}}}.
\end{equation}
Repeating the proof of Theorem 5 in ~\cite{BhaPip} word by word but replacing the shape function (denoted as $q$ in ~\cite[p. 7]{BhaPip}) with the lattice function $l:V_{\RR}^{(0)}\rightarrow X_{n-1}$ defined for every $v=t(v)(g_1(v),g_2(v))v_0\in \RR_+\cdot \Sigma_{n-1}\times\Sigma_{r-1}v_0$ by:
\begin{equation}
	l(v)=g_0g_1(v)^tg_0^{-1}\Lambda_w,
\end{equation}
we get:
\begin{equation}\label{17:45}
	\abs{S_{T,\ZZ}}=\operatorname{Vol}(S_T)+o(T^d),\abs{L_{T,\ZZ}}=\operatorname{Vol}(L_T)+o(T^d).
\end{equation}
Given a subset $B\subset Y_{n-1}$ and $T>0$ denote $\mu_T(B)$ to be the proportion of totally real number fields $\FF$ with discriminant less than $T$ that satisfy $\Gamma_{E}^{\pi}(\FF)\in B$ (recall Subsection \ref{NLN}). Invoke Lemma \ref{structure} to deduce that for any $T>0$:
\begin{equation}\label{ha}
	\mu_T(A)=\frac{\abs{S_{T,\ZZ}\cap f_{E}^{-1}(V)}}{\abs{L_{T,\ZZ}}}=\frac{\abs{S_{T,\ZZ}\cap f_{E}^{-1}(V)}}{\abs{S^{\eps}_{T,\ZZ}\cap f_{E}^{-1}(V)}}\cdot\frac{\abs{S^{\eps}_{T,\ZZ}\cap f_{E}^{-1}(V)}}{\abs{S_{T,\ZZ}^{\eps}}}\cdot\frac{\abs{S_{T,\ZZ}^{\eps}}}{\abs{S_{T,\ZZ}}}\cdot\frac{\abs{S_{T,\ZZ}}}{\abs{L_{T,\ZZ}}}.
\end{equation}
We estimate every element in the above product separately. First, we deal with $	p_T^{\eps}(V)=\frac{\abs{S_{T,\ZZ}^{\eps}\cap f_E^{-1}(V)}}{\abs{S_{T,\ZZ}^{\eps}}}$. By application of Lemma \ref{calc} for the OBJOM set $\left(S\times \Sigma_{r-1}^{\eps}\right)v_0$:
\begin{equation}\label{18:14}
	p_T^{\eps}(V)\rightarrow \operatorname{Vol}_{n-1}(V)\text{ as }T\rightarrow \infty.
\end{equation} 
By Equation (\ref{17:45}):
\begin{equation}
	\frac{\abs{S_{T,\ZZ}}}{\abs{L_{T,\ZZ}}}\rightarrow \frac{\operatorname{Vol}(S_1)}{\operatorname{Vol}(L_1)}\text{ as }T\rightarrow \infty
\end{equation}
and by the calculation in the proof of ~\cite[Proposition 12]{BhaPip} applied for $l^{-1}(S)\cap \{\operatorname{Disc}(v)\leq 1\}$ instead of $\mathcal R_{1,W}$:
\begin{equation}\label{ONE}
	\frac{\abs{S_{T,\ZZ}}}{\abs{L_{T,\ZZ}}}\rightarrow m_{X_{n-1}}(S)\text{ as }T\rightarrow \infty
\end{equation}
(recall Subsection \ref{NLN} for definition of $m_{X_{n-1}}$, the Haar measure on the space of lattices).
Since $S_{1}^{\eps}$ is bounded we can use ~\cite[Lemma 6]{BhaPip} on it to deduce that for any $T>0$:
\begin{equation}
	\abs{S_{T,\ZZ}^{\eps}}=\operatorname{Vol}(S_{T}^{\eps})+o(T^d)
\end{equation}
which implies, by Equations (\ref{17:45}) and (\ref{18:03}) that:
\begin{equation}\label{TWO}
	\frac{\abs{S_{T,\ZZ}}}{\abs{S^{\eps}_{T,\ZZ}}}=	1-\eps+o_T(1).
\end{equation}
Consequently, by dividing the space into its disjoint intersections with $f^{-1}_E(V),f^{-1}_E(V^c)$:
\begin{equation*}
	\left(\abs{S_{T,\ZZ}\cap f_E^{-1}(V)}-\abs{S^{\eps}_{T,\ZZ}\cap f_E^{-1}(V)}\right)+\left(\abs{S_{T,\ZZ}\cap f_E^{-1}(V^c)}-\abs{S^{\eps}_{T,\ZZ}\cap f_E^{-1}(V^c)}\right)\leq \eps \abs{S_{T,\ZZ}^{\eps}}+\abs{S_{T,\ZZ}^{\eps}}o_T(1)
\end{equation*}
which implies, since both terms in parenthesis in the above equation are non-negative:
\begin{equation}
	\frac{\abs{S_{T,\ZZ}\cap f_E^{-1}(V)}}{\abs{S^{\eps}_{T,\ZZ}\cap f_E^{-1}(V)}}-1\leq \eps\frac{\abs{S_{T,\ZZ}^{\eps}}}{\abs{S_{T,\ZZ}^{\eps}\cap f_E^{-1}(V)}}+\frac{\abs{S_{T,\ZZ}^{\eps}}}{\abs{S_{T,\ZZ}^{\eps}\cap f_E^{-1}(V)}}o_T\left(1\right)=\eps p^{\eps}_T(V)+p_T^{\eps}(V)o_T(1)
\end{equation}
but now using (\ref{18:14}):
\begin{equation}\label{THREE}
	\frac{\abs{S_{T,\ZZ}\cap f_E^{-1}(V)}}{\abs{S^{\eps}_{T,\ZZ}\cap f_E^{-1}(V)}}=1+\eps\operatorname{Vol}(V)(1+o_T(1)).
\end{equation}
Using (\ref{ONE}),(\ref{TWO}),(\ref{18:14}) and (\ref{THREE}) and taking $\eps\rightarrow 0$ we deduce:
\begin{equation}\label{14:15}
	\mu_T(A)\rightarrow m_{X_{n-1}}(S)\operatorname{Vol}(V)=m_{Y_{n-1}}(A)\text{ as }T\rightarrow\infty
\end{equation}
where the last equality holds by Fubini's Theorem and Definition \ref{BS}. To deduce the above limit for any other Jordan measurable subset $A\subset Y_{n-1}$ invoke Theorem \ref{finish}. Note that this proves the desired claim, which is about the weak convergence of $\mu_T$ to $\mu_{Y_{n-1}}$. 
\end{proof}

\begin{proof}[Proof of Theorem \ref{torsions2g}]
	The reader may find it useful to recall the definition of the corresponding grid which appears in Subsection \ref{NLN}. By Lemma \ref{structure}, there exists a fixed basis of $V_0$ given by $w=(w_j)_{j=1}^{n-1}$, a function $\Pi:\mathcal F_{\infty}^{(0)}\rightarrow S_n$ (recall Subsection \ref{Not}) and $g_0\in \operatorname{SL}_{n-1}(\RR)$ such that:
	\begin{equation}
		\Sigma_{n-1}\vcentcolon=g_0^t\Sigma_{n-1}'g_0^{-t}
	\end{equation}
satisfies that for any $\Lambda_w$-$\Sigma_{n-1}$-basic subset $S\times_{\Sigma_{n-1}} U\subset Y_{n-1}$:
	\begin{equation}\label{0:51}
		\mathcal F^{(0)}_X\cap (\Gamma_{V_0}^{\Pi})^{-1}(S\times_{\Sigma_{n-1}} U)\xleftrightarrow{1:1}
	\end{equation}
	\begin{equation*}
		\{\text{irreducible, maximal points inside }[0,X] (g_0^{-t}S^{-1}g_0^{t}\times \Sigma_{r-1})v_0\}\cap ((f_w+q_0)\operatorname{mod}\Lambda_{w})^{-1}(U_w)
	\end{equation*}
	where $q_0\in V_0$ is some fixed vector (recall Definition \ref{UW} for definition of $U_w$). In this case:
	\begin{equation}\label{0:50}
		f_w(p)=\frac{1}{\inner{\overline 1}{\overline 1}}\sum_{i=1}^{n-1}\inner{\overline 1}{\overline{\alpha}_i(g_pv_0)}w_i=\frac{1}{n}\sum_{i=1}^{n-1}\inner{\overline 1}{\overline{\alpha}_i(g_pv_0)}w_i.
	\end{equation}
By the discussion carried in ~\cite[p.35 for $n=3$,p.41 for $n=4$ and p.50 for $n=5$]{PipThes} for every $i=1,\dots,n-1$:
\begin{equation}\label{0:49}
	\inner{\overline 1}{\overline{\alpha}_i(g_pv_0)}\in \ZZ\text{ if }p\in V_{\ZZ}\text{ and moreover }\inner{\overline 1}{\overline{\alpha}_i(g_pv_0)}\text{ is a polynomial in the coefficients of }p. 
\end{equation}
Let $\Lambda_w=\operatorname{span}_{\ZZ}\{w_1,\dots,w_{n-1}\}$ and denote:
\begin{equation}
	D=\{\sum_{i=1}^{n-1}\frac{j_i}{n}w_i\operatorname{mod}\Lambda_w:j_i=1,\dots,n \}.
\end{equation}
By Equations (\ref{0:51})-(\ref{0:49}) for every totally real $S_n$-number field $\FF$ of degree $n$:
\begin{equation}\label{15:24}
	\Gamma_{V_0}^{\Pi}(\FF)\in \Sigma_{n-1}\times_{\Sigma_{n-1}}D.
\end{equation}
Let $m_{\Sigma_{n-1}}$ be the restriction of $m_{\operatorname{SL}_{n-1}(\RR)}$ to $\Sigma_{n-1}$ and let $k$ be a divisor of $n$. Recall Definition \ref{torDef} to note that if $\zeta\in D$ is such that $k\zeta\in \Lambda_{w}$ and $(k-1)\zeta\notin \Lambda_{w}$ then:
\begin{equation}
	\on{SL}_{n-1}(\RR)\times_{\Sigma_{n-1}}\{\zeta\}=Y_n(k).
\end{equation}
Denote also
\begin{equation}
	F_{\zeta}:=\overline{\Sigma_{n-1}\times_{\Sigma_{n-1}}\{\zeta\}}\subset Y_n(k).
\end{equation}
By Equation (\ref{0:49}) $f_i(p)\vcentcolon=\inner{\overline 1}{\overline{\alpha}_i(g_pv_0)}$ is an integer polynomial in the coefficients of $p$. Therefore there exists a congruence set $C(\zeta)$ with congruences modulo $k$ such that:
\begin{equation}\label{key}
	f_w(p)=\zeta\iff p\in C(\zeta)\iff \Gamma_{V_0}^{\Pi}(p)\in F_{\zeta}.
\end{equation}
Let $\mu_T$ be defined for any $T>0$ as in the proof of Theorem \ref{unif3}. To prove Theorem \ref{torsions2g} it will suffice to find $\alpha_{\zeta}>0$ such that for any OBJOM $S\subset \Sigma_{n-1}$, $A_{\zeta,S}=S\times_{\Sigma_{n-1}}\{\zeta\}$ satisfies:
\begin{equation}
	\mu_T(A_{\zeta,S})\rightarrow \alpha_{\zeta}m_{X_{n-1}}(S)=m_{Y_n(k)}\mid_{F_{\zeta}}(A_{\zeta,S})\text{ as }T\rightarrow \infty
\end{equation}
by (\ref{key}). Let $\alpha_{\zeta}$ be the density of $C(\zeta)$ inside $V_{\ZZ}$, namely:
\begin{equation}
	\alpha_{\zeta}=\lim_{T\rightarrow \infty}\frac{\abs{C(\zeta)\cap V_{\ZZ}\cap B_T(0)}}{T^d}
\end{equation}
which also satisfies since $S_T^{\eps}$ is an OBJOM:
\begin{equation}\label{15:336}
	\frac{\abs{S^{\eps}_{T,\ZZ}\cap f_{w}^{-1}(\{\zeta\})}}{\abs{S_{T,\ZZ}^{\eps}}}=\frac{\abs{S^{\eps}_{T,\ZZ}\cap C(\zeta)}}{\abs{S_{T,\ZZ}^{\eps}}}\rightarrow \alpha_{\zeta}\text{ as }T\rightarrow \infty
.\end{equation}
 Then precisely as in Equation (\ref{ha}) it holds that:
\begin{equation}\label{ha'}
	\mu_T(A_{\zeta,S})=\frac{\abs{S_{T,\ZZ}\cap f_{w}^{-1}(\{\zeta\})}}{\abs{L_{T,\ZZ}}}=\frac{\abs{S_{T,\ZZ}\cap f_{w}^{-1}(\{\zeta\})}}{\abs{S^{\eps}_{T,\ZZ}\cap f_{w}^{-1}(\{\zeta\})}}\cdot\frac{\abs{S^{\eps}_{T,\ZZ}\cap f_{w}^{-1}(\{\zeta\})}}{\abs{S_{T,\ZZ}^{\eps}}}\cdot\frac{\abs{S_{T,\ZZ}^{\eps}}}{\abs{S_{T,\ZZ}}}\cdot\frac{\abs{S_{T,\ZZ}}}{\abs{L_{T,\ZZ}}}.
\end{equation}
where the second term converges, by Equation (\ref{15:336}), to $\alpha_{\zeta}$ as $T\rightarrow \infty$ and the remaining terms are the same as in the proof of Theorem \ref{unif3}. Therefore $\mu_T(A_{\zeta,S})$ converges to $\alpha_{\zeta}m_{X_{n-1}}(S)$ as $T\rightarrow \infty$. 
We therefore choose:
\begin{equation}
	\{F^{(k)}_{\ell}\}=\{F_{\zeta}:\zeta\in D\text{ is such that }k\zeta\in \Lambda_w,(k-1)\zeta\notin\Lambda_w\}
\end{equation}
and conclude by Equations (\ref{15:24}) and (\ref{key}).
\end{proof}
\appendix
\section{Calculations}\label{A}
\subsection{Multiplication Tables}\label{muls}
For convenience we include explicitly the multiplication tables that define Bhargava's correspondence in ~\cite{Bha1,Bha2}. They are copied from ~\cite[p. 35,39-40]{PipThes}.
\subsubsection{Case $n=3$}
Let $v\in V_{\RR}$ be given by $v=(a,b,c,d)$. The multiplication table corresponding to $v$ is:
\begin{equation*}
	\alpha_1\alpha_2=-ad
\end{equation*}
\begin{equation}\label{MT3}
	\alpha_1^2=-ac+b\alpha_1-a\alpha_2
\end{equation}
\begin{equation*}
	\alpha_2^2=-bd+d\alpha_1-c\alpha_2.
\end{equation*}
\subsubsection{Case $n=4$}
Let $v\in V_{\RR}$ be given by $v=(a_{11}x^2 + a_{22}y^2 + a_{33}z^2 + a_{12}xy + a_{13}xz + a_{23}yz, b_{11}x^2 + b_{22}y^2 + b_{33}z^2 + b_{12}xy + b_{13}xz + b_{23}yz)$. The multiplication table corresponding to $v$ is:
\begin{equation*}
	\alpha_1^2=h_{11}+g_{11}\alpha_1+f_{11}\alpha_2+e_{11}\alpha_3
\end{equation*}
\begin{equation*}
	\alpha_2^2=h_{22}+g_{22}\alpha_1+f_{22}\alpha_2+e_{22}\alpha_3
\end{equation*}
\begin{equation}\label{MT4}
	\alpha_3^2=h_{33}+g_{33}\alpha_1+f_{33}\alpha_2+e_{33}\alpha_3
\end{equation}
\begin{equation*}
	\alpha_1\alpha_2=h_{12}+e_{12}\alpha_3
\end{equation*}
\begin{equation*}
	\alpha_1\alpha_3=h_{13}+f_{13}\alpha_2+e_{13}\alpha_3
\end{equation*}
\begin{equation*}
	\alpha_2\alpha_3=h_{23}+g_{23}\alpha_1+f_{23}\alpha_2+e_{23}\alpha_3.
\end{equation*}
\subsubsection{Case $n=5$}
Let $v\in V_{\RR}$ be given by $v=(A_1,\dots,A_4)$ where $A_i$ are $5\times 5$ skew symmetric matrices. The multiplication table corresponding to $v$ is:
\begin{equation}\label{MT5}
	\alpha_i\alpha_j=c_{ij}^0+\sum_{k=1}^4c^k_{ij}\alpha_k
\end{equation}
for some constants $c_{ij}^k$ which are some polynomials in the entries of $(A_1,\dots,A_4)$ and appearing in ~\cite[p. 51]{PipThes}.
\subsection{Construction of a Basis}\label{basCons}
Let $n=3,4,5$ and fix $v_0\in V_{\RR}^{(0)}$. Let $\overline\alpha_1(v_0),\dots,\overline\alpha_n(v_0)\in \RR^n,\overline \alpha_n(v_0)=\overline 1$ be a set of vectors such that $\operatorname{MT}_{v_0}(\overline{\alpha}_1,\dots,\overline{\alpha}_{n-1})$ with the pointwise product on $\RR^n$.

For any $g=\begin{pmatrix}
	\mid\ & \cdots &\mid \\
	\overline{g_1} & \cdots & \overline{g_{n-1}}\\
	\mid & \cdots & \mid
	\end{pmatrix}$ and $h\in \operatorname{SL}_{r-1}(\RR)$, write $v=t(g,h)v_0;t\in \RR$ in a unique way (it is indeed unique by ~\cite[Section 3]{BhaPip}) and for any $i=1,\dots,n-1$, denote:
\begin{equation}\label{nirit1}
	\overline\alpha_{i,V_0}(v)=t\inner{\overline{g_i}}{P_{V_0}\overline\alpha_i(v_0)}.
\end{equation}
We know that there exist vectors $\overline{\alpha}_1(v),\dots,\overline{\alpha}_{n-1}(v)\in \RR^n$ such that $\operatorname{MT}_v(\overline{\alpha}_1(v),\dots,\overline{\alpha}_{n-1}(v))$ and:
\begin{equation}\label{nirit2}
	P_{V_0}\overline{\alpha}_i(v)=\overline\alpha_{i,V_0}(v)\text{ for }i=1,\dots,n-1
\end{equation}
using the last sentence of ~\cite[Theorem 3]{BhaPip} and as is used implicitly in ~\cite[p. 39 for $n=3$, p.44 for $n=4$ and p.52 for $n=5$]{PipThes}. Therefore by Equations (\ref{nirit1}),(\ref{nirit2}):
\begin{equation}\label{14:49}
	\operatorname{span}_{\ZZ}\{P_{V_0}\overline \alpha_1(v),\dots,P_{V_0}\overline \alpha_{n-1}(v)\}=g_0g^tg_0^{-1}	\operatorname{span}_{\ZZ}\{P_{V_0}\overline \alpha_1(v_0),\dots,P_{V_0}\overline \alpha_{n-1}(v_0)\}
\end{equation}
where:
\begin{equation*}
	g_0=\begin{pmatrix}
	\mid\ & \cdots &\mid \\
	P_{V_0}\overline \alpha_1(v_0) & \cdots & P_{V_0}\overline \alpha_{n-1}(v_0)\\
	\mid & \cdots & \mid
	\end{pmatrix}.
\end{equation*}
For any $i=1,\dots,n-1$ write:
\begin{equation}\label{veryLate}
	\overline{\alpha}_i(v)=P_{V_0}\overline{\alpha}_i(v)+\operatorname{trace}(\overline{\alpha}_i(v))\overline 1.
\end{equation}
Since the action of $\RR\cdot \operatorname{SL}_{n-1}\times\operatorname{SL}_{r-1}(\RR)$ on $V_{\RR}$ is smooth and since $v\mapsto \operatorname{trace}(\overline{\alpha}_i(v))$ is smooth by direct computation of the trace in the multiplication table (see \ref{muls}), we deduce that the functions $v\mapsto\overline{\alpha}_i(v)$ are smooth for $i=1,\dots,n-1$.
\subsection{Additional Lemmas}
In this part of the Appendix we state several supporting results for Subsection \ref{ver}.
\begin{Lemma}\label{lesles}
	Let $d\in \NN$. Let $S$ be a level surface of some smooth homogeneous function (as in Theorem \ref{myWeyl}) on $\RR^d$ and let $A\subset S$ be bounded, open and Jordan measurable in $S$. Let $F:A\rightarrow \RR$ be smooth. If $\nabla F(x)\neq 0$ for $m_A$ a.e. $x\in A$ (where $m_A$ is the surface measure on $A$) then $F_*m_A\ll\lambda_{\RR}$ where $\lambda_{\RR}$ is the Lebesgue measure on $\RR$.
\end{Lemma}
\begin{proof}
	We first prove the claim when there exists $c_0>0$ such that $\norm{\nabla F}>c_0$ on $A$ (and therefore on $\overline{A}$ because $F$ is smooth). In this case, by compactness of $\overline A$ we can find $c>0$ and $\eps_0>0$ such that for any $\eps<\eps_0$ it holds that $\lambda_{\RR}(F(B_{\eps}^A(x)))\geq c\norm{\nabla F(x)}m_A(B_{\eps}^A(x))\geq cc_0m_A(B_{\eps}^A(x))$ where $B_{\eps}^A$ is an (e.g. Euclidean) open ball of radius $\eps$ in $A$. By standard estimation argument, it will hold that $F_*m_A\ll \lambda_{\RR}$. Next we prove the Lemma's claim in the general context. Define the closed subset of $A$ by $E=\{x\in A:\nabla F(x)=0\}$ and let $E_n\supset E$ be a sequence of open subsets of $A$ such that $m_A(E_n)\rightarrow 0$ as $n\rightarrow \infty$. For any $n$, denote $\mu_n=\mathbbm{1}_{E_n^c}m_A$ to be a sequence of measures on $A$ such that $\mu_n\rightarrow m_A$ weakly as $n\rightarrow \infty$. Let $S\subset \RR$ be a $\lambda_{\RR}$-null set. By compactness of $\overline A\cap E_n^c$, there exists $c_n>0$ such that $\norm{\nabla F}\geq c_n$ on $\overline A\cap E_n^c$ so by what we proved before: $F_*\mu_n\ll\lambda_{\RR}$ for every $n$. Since $F_*\mu_n\rightarrow F_*m_{A}$ weakly as $n\rightarrow \infty$ we deduce $F_*m_A\ll\lambda_{\RR}$ as desired. 
\end{proof}
\begin{Lemma}\label{reA}
	Let $n=3,4,5$, $d=4,12,40$, respectively and let $U\subset V^{(0)}_{\RR}$ be an open subset. Assume that $\alpha_1,\dots,\alpha_{n-1}:U\rightarrow \RR^n$ are smooth and satisfy (\ref{MT3}) if $n=3$, (\ref{MT4}) if $n=4$ and (\ref{MT5}) if $n=5$. Then $\alpha_1,\dots,\alpha_{n-1}$ are real analytic on $U$.
\end{Lemma}
\begin{proof}
Let $v_0\in U$ and denote $\phi:\RR\times\operatorname{SL}_{n-1}(\RR)\times\operatorname{SL}_{r-1}(\RR)\rightarrow V^{(0)}_{\RR}$ the  isomorphism given by $(t,g,h)\mapsto t(g,h)v_0$. By its definition, the action of $\RR\times\operatorname{SL}_{n-1}(\RR)\times\operatorname{SL}_{r-1}(\RR)$ on $V_{\RR}^{(0)}$ is by polynomials in the coefficients of $(t,g,h)\in \RR\times\operatorname{SL}_{n-1}(\RR)\times\operatorname{SL}_{r-1}(\RR)$ (see ~\cite[p. 38,47,52]{PipThes} or ~\cite{Bha1},~\cite{Bha2} for definition of this action) so that the map $\phi$ is real analytic. Therefore, by the inverse functions theorem for real analytic functions (see ~\cite[Theorem 1.8.1]{bookAn}), so is the map $\phi^{-1}$. By Equation (\ref{veryLate}) it is evident that for any $i=1,\dots,n-1$ the function $\overline{\alpha}_i\circ\phi$ is real analytic. Indeed, the equation says:
\begin{equation}
		\overline{\alpha}_i(\phi(t,g,h))=\overline{\alpha}_i(t(g,h)v_0)=P_{V_0}\overline{\alpha}_i(t(g,h)v_0)+\operatorname{trace}(\overline{\alpha}_i(t(g,h)v_0))\overline 1=t\inner{\overline{g_i}}{P_{V_0}\overline\alpha_i(v_0)}+P(t,g,h)
\end{equation}
where $g=\begin{pmatrix}
	\mid\ & \cdots &\mid \\
	\overline{g_1} & \cdots & \overline{g_{n-1}}\\
	\mid & \cdots & \mid
	\end{pmatrix}$ and $P$ is some polynomial in $t$ and the entries of $g,h$ by direct calculation from \ref{muls} to find $\operatorname{trace}(\overline{\alpha}_i(v))$ in terms of $v$ and ~\cite[p. 38,47,52]{PipThes} to write $v=t(g,h)v_0$ in terms of $t,g,h$. Altogether we deduce that the function:
\begin{equation}
	\overline{\alpha}_i=(\overline{\alpha}_i\circ\phi)\circ\phi^{-1}
\end{equation}
is the composition of two real analytic functions and therefore real analytic.
\end{proof}
\begin{Lemma}[~\cite{bookAn}, Chapter 3]\label{posM}
	Let $d\in \NN$. If a real analytic function $F:U\subset \RR^d\rightarrow \RR$ vanishes on a set of positive Lebesgue measure in the open set $U\subset \RR^n$ then $F\mid_U=0$ identically. 
\end{Lemma}

\section{Some Topology}\label{B}
For this Section, fix a lattice $\Lambda=\operatorname{sp}_{\ZZ}\{w_1,\dots,w_n\}\subset \RR^n$ and a fundamental domain $\Sigma$ for the action of $\operatorname{stab}_{\operatorname{SL}_n(\RR)}(\Lambda)$ on $\operatorname{SL}_n(\RR)$. We abbreviate '$\Sigma$-$\Lambda$ basic subsets' to 'basic subsets'. Also, for this Section only we denote $\mu\vcentcolon=m_{Y_{n-1}}.$
\begin{Lemma}\label{9}
	Let $W\subset Y_{n}$ be a measurable set with $\mu(\partial W)=0$. Then for every $\eps>0$ there exist $B_1,\dots,B_m\subset W$ disjoint basic subsets, such that $\sum_{i=1}^m\mu(B_i)\geq \mu(W)-\epsilon$.
\end{Lemma}
\begin{proof}
	We may assume that $W$ is bounded, since otherwise we may intersect $W$ with large enough bounded ball in $Y_n$. Let $U\supset \partial W$ be open such that $\mu(U)\leq \epsilon$. For any $x\in \overline{W\setminus U}$, let $B_x\subset \text{int}(W)$ be a Basic subset. Take a finite sub-cover of $\overline{W\setminus U}$, $(B_i)_{i=1}^m$. Then clearly the $B_i$'s are included in $W$ (they are not necessarily disjoint), and:
	\begin{equation*}
		\mu\left(\bigcup_{i=1}^mB_i\right)\geq \mu(W)-\mu(W\cap U)\geq \mu(W)-\mu(U)\geq \mu(W)-\epsilon.
	\end{equation*}
	Finally, note that the intersection of every two basic subsets can be estimated by finite union of disjoint basic subsets with arbitrarily small defect, so we may also assume that the $B_i$'s are disjoint. 
\end{proof}
The following Lemma is a Corollary of Lemma \ref{Bas} and the definition of basic subsets.
\begin{Lemma}\label{dani}
	Let $S\subset Y_{n}$ be bounded and satisfy $\mu(S)\leq \epsilon$ for some $\eps>0$. Then there exist basic subsets $(D_i)_{i=1}^{k}$ such that $S\subset \bigcup _{i=1}^kD_i$ and $\sum_{i=1}^{k}\mu(D_i)\leq 2\epsilon.$
\end{Lemma}
\begin{theorem}\label{finish}
	Let $W\subset Y_{n-1}$ be measurable with $\mu$-zero boundary and let $E$ be as in the proof of Theorem \ref{unif3} (see \ref{PF}). Then $\lim_{T\rightarrow \infty}\mu_T(W)=\mu(W)$.
\end{theorem}
\begin{proof}
	We may assume that $W$ is bounded, since otherwise we can intersect $W$ with a large enough bounded ball in $Y_{n-1}$. Let $\eps>0$ and using the Lemma \ref{9}, find $B_1,\dots,B_m$, a collection of disjoint basic subsets, satisfying the conclusion of the Lemma. Then by Lemma \ref{9} and Lemma \ref{dani} applied for $W\setminus \bigcup_{i=1}^{m}B_i$ resulting in a collection $(D_i)_{i=1}^{k}$ of basic subsets, we get:
	\begin{equation*}
		\abs{\mu(W)-\mu_T(W)}\leq \abs{\mu(\bigcup_{i=1}^mB_i)-\mu_T(\bigcup_{i=1}^mB_i)}+\abs{\mu(W\setminus \bigcup_{i=1}^mB_i)-\mu_T(W\setminus \bigcup_{i=1}^mB_i)}
	\end{equation*}
	\begin{equation*}
		\leq \abs{\mu(\bigcup_{i=1}^mB_i)-\mu_T(\bigcup_{i=1}^mB_i)}+\epsilon+\mu_T(\bigcup_{i=1}^{k}D_i)\leq \abs{\mu(\bigcup_{i=1}^mB_i)-\mu_T(\bigcup_{i=1}^mB_i)}+\sum_{i=1}^k\mu_T(D_i)+\epsilon,
	\end{equation*}
	implying, by Lemma \ref{9}, Equation (\ref{14:15}) and since the $B_i$'s are disjoint, that for any $\eps>0$, $\limsup_{T\rightarrow \infty} \abs{\mu(W)-\mu_T(W)}\leq 3\epsilon$, as required.
\end{proof}
\newpage
\bibliographystyle{plain}
\bibliography{BibErg}{}

\end{document}